\documentclass[a4paper,11pt,reqno]{article}

\usepackage{comment}
\usepackage{latexsym}
\usepackage{amssymb,amsbsy,amsmath,amsfonts,amssymb,amscd,amsthm}
\usepackage{accents}	
\usepackage{graphicx}
\graphicspath{{./}{Figure/}}
\usepackage{hyperref}
\usepackage{dsfont}
\usepackage{mathtools}
\usepackage{pdfsync}
\usepackage{epsfig,epstopdf}
\usepackage{caption}
\usepackage{xcolor}
\usepackage{hyperref}
\usepackage[title]{appendix}
\usepackage{subcaption}
\usepackage{float}
\usepackage{color}
\usepackage{verbatim}
\usepackage{dsfont}
\usepackage[inline]{enumitem}
\usepackage{hyperref}

\usepackage{geometry}

\usepackage{enumitem}

\newtheorem{prop}[]{Proposition}
\newtheorem{rmrk}[]{Remark}
\newtheorem{lem}[]{Lemma}
\newtheorem{thrm}[]{Theorem}

\newcommand{\R}{\mathbb{R}}
\newcommand{\X}{\mathcal{X}}
\newcommand{\V}{\mathcal{V}}

\let\t=\tilde

\newcommand{\beq}{\begin{equation}}
\newcommand{\eeq}{\end{equation}}
\newcommand{\beqa}{\begin{eqnarray}}
\newcommand{\eeqa}{\end{eqnarray}}

\newcommand{\di}{\, {\rm d}}
\newcommand {\e}  {\varepsilon}
\newcommand{\pa}{\partial}

\numberwithin{equation}{section}
\numberwithin{prop}{section}
\numberwithin{ass}{section}
\numberwithin{rmrk}{section}

\title{Optimizing vaccine allocation in an age-structured SIR model}

\author{Luis Almeida\thanks{Sorbonne Université, Université Paris Cité, CNRS, Laboratoire de Probabilités, Statistique et Modélisation, LPSM, F-75005 Paris, France}
\and
Romain Ducasse \thanks{Université Paris Cité and Sorbonne Université, CNRS, LJLL, F-75005 Paris, France} \and
Elisa Paparelli\thanks{Sorbonne Université, Université Paris Cité, CNRS, Laboratoire Jacques-Louis Lions, LJLL, F-75005 Paris, France}
}

\begin{document}
\date{}
\maketitle



\begin{abstract}
 We study an optimal control problem where the objective is to find the best vaccine allocation during an epidemic outbreak. The epidemic dynamics is described by an age-structured SIR model with nonlocal interactions. Both the infection and death rates depend on the age of the individuals, reflecting the effect of heterogeneities within the population.

 Our model includes a vaccination term, depending on time and age, which serves as a control function. The aim is to minimize the impact of the epidemic, that is, the number of casualties, under the constraint of limited vaccine supply.

 In a first part, we show that our optimization problem is equivalent to another {\em static} optimization problem. We then use this new optimization problem to obtain qualitative properties for the optimal allocations of vaccines.
\end{abstract}

\section{Introduction}

\subsection{The SIR model with vaccination}\label{sec21}

In $1927$, in a series of celebrated papers, Kermack and McKendrick \cite{kermack1927, kermack1932, kermack1933}, introduced several deterministic models designed to describe the temporal development of a disease in a population, among which the celebrated \emph{SIR} model, which became a cornerstone of mathematical epidemiology. In this model, the population under consideration is divided into three {\em compartments}: the \emph{susceptible} (who are not infected), the {\em infectious} (who have the disease and can transmit it to the susceptible) and the {\em recovered/deceased} (who had the disease and are now immune, or dead - in  any case, they do not play any role in the propagation of the disease). \\

In the present paper, we consider an age-structured model in which each age group exhibits distinct susceptibility and mortality characteristics, and where a vaccination strategy is applied. We study the optimal control problem aiming at finding the best way to allocate vaccines in order to minimize the outcome of the epidemic.\\

The model we consider is the following system of integro-differential equations
\begin{equation}\label{syst}
\begin{cases}
     \displaystyle{\partial_t S(t,x) = -S(t,x)\int_{\X} \beta(x,y)\,I(t,y)\,\di y - \nu(t,x),\quad \quad t>0,\ x\in \X}, \\
     \displaystyle{\partial_t I(t,x) = S(t,x)\int_{\X} \beta(x,y)\,I(t,y)\,\di y - \mu(x)I(t,x),\quad t>0,\ x\in \X},
\end{cases}
\end{equation}
with given initial data that represent the initial distribution of populations:
$$
S(0,x)\coloneqq S_0(x),\quad I(0,x)\coloneqq I_0(x).
$$
The functions $S(t,x), I(t,x)$ are the number of susceptible and infectious individuals respectively, at time $t>0$ and age $x$. We take $x \in \X$, where $\X$ is a compact interval in $\R$, say $\X = [0,A]$ for instance, for some $A>0$.\\

The function $\beta(x,y)$ represents the {\em rate of transmission/infection} when a susceptible individual with age $x$ meets an infected individual with age $y$. The larger it is, the more likely the susceptible individual is infected. The quantity $\beta$ accounts both for the likelihood that a contact between individuals aged $x$ and $y$ occurs and for the likelihood that this leads to a contamination. We mention that a priori $\beta$ need not be symmetric.\\

The quantity $\mu(x)$ is the mortality rate of infected individuals aged $x$. The larger $\mu(x)$, the faster the individuals aged $x$ die after their infection. Once they are dead, these individuals do not play any role in the dynamic (they can not contaminate anymore).\\

The function $\nu(t,x)$ represents the rate of vaccination of the susceptible aged $x$ at time $t>0$.\\

The first goal of the present paper is to study {\em how to choose the vaccination strategy $\nu$ in order to minimize the impact of the epidemic}, which is here chosen to be measured by the number of deaths. We shall call this the {\bf Optimal Vaccination Problem}, (OVP). It is stated rigorously below as \eqref{ocp}.

Since the direct application of general control theory tools is not straightforward in this context, we will show that the Optimal Vaccination Problem is equivalent to what we shall call the {\bf Initial Vaccination Problem} (IVP), stated below as \eqref{ocp:equivalent}. In the Initial Vaccination Problem, we can vaccinate a given number of individuals {\em before} the disease starts to propagate. The problem consist in finding which individuals should be vaccinated in order to minimize the epidemic outcome. The difference between these two problems is that in the Optimal Vaccination Problem, we can vaccinate the individuals at any time, whereas in the Initial Vaccination Problem, we vaccinate only at the start of the epidemic.\\

The equivalence between the (OVP) and the (IVP) will be established in the first part of the paper. The proofs rely on rewriting the SIR model as a nonlinear integral equation and on the careful use of comparison principles.

In the second part of the paper, we completely solve in a specific case the (IVP), and the (OVP), using their equivalence.\\

Our results will allow us to obtain qualitative properties on the model, see Remark \ref{req quali} below. Roughly speaking, we shall see that an optimal strategy for the (OVP) is to vaccinate {\em as soon as possible} and to target all the individuals in some specific age class.

\begin{rmrk}
    It is sometimes convenient to add in the system \eqref{syst} two other equations that describe the evolution of the number of vaccinated and deceased individuals with age $x$ at time $t>0$, denoted respectively $V(t,x)$ and $R(t,x)$. The full system would be
\begin{equation*}
\left\{
\begin{array}{rll}
     \partial_t S(t,x) &= -S(t,x)\int_{\X} \beta(x,y)\,I(t,y)\,\di y - \nu(t,x),\quad &t>0,\ x\in \X, \\
     \partial_t I(t,x) &= S(t,x)\int_{\X} \beta(x,y)\,I(t,y)\,\di y - \mu(x)I(t,x),\quad &t>0,\ x\in \X,\\
     \partial_t V(t,x) &= \nu(t,x),\quad &t>0,\ x\in \X,\\
     \partial_t R(t,x) &= \mu(x)I(t,x),\quad &t>0,\ x\in \X.
\end{array}
\right.
\end{equation*}
We have the following graphical representation of the evolution, showing all the compartments.
\begin{center}
    \includegraphics[scale=0.5]{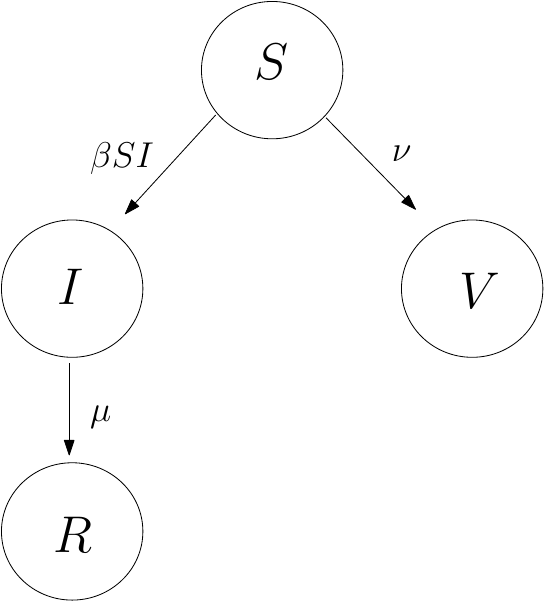}
\end{center}
At time $t=0$, almost all individuals are susceptible, some are infectious (usually, a very small amount). Then, the individuals move from one compartment to another.\\

Observe that, for each age $x$, the total population $P(t,x) \coloneqq S(t,x)+I(t,x)+R(t,x)+V(t,x)$ is conserved, the individuals only move from one compartment to another during the evolution, they never disappear. 
\end{rmrk}

\subsection{The SIR model}\label{sec sir classical}

Before diving into the study of our model \eqref{syst}, let us mention some salient results concerning the ``classical" SIR system, that is, without the vaccination term $\nu$. This will be the occasion to recall some technical facts that we shall use in our proofs.\\

When there is no vaccination in \eqref{syst}, we have the classical SIR model with non-local contamination
\begin{equation}\label{SIR original}
\left\{
\begin{array}{rll}
     \displaystyle{\partial_t S(t,x)} &= \displaystyle{-S(t,x)\int_{\X} \beta(x,y)\,I(t,y)\,\di y} , &t>0,\ x\in \X, \\
     \displaystyle{\partial_t I(t,x)} &= \displaystyle{S(t,x)\int_{\X} \beta(x,y)\,I(t,y)\,\di y - \mu(x)I(t,x)}, &t>0,\ x\in \X.
\end{array}
\right.
\end{equation}
This model was first introduced by Kendall \cite{kendall1965mathematical}, as a spatial generalization of the SIR model of Kermack and McKendrick \cite{kermack1927}: the variable $x$ was meant to be a position in a physical space. Several authors studied the {\em spreading properties} of this model, see \cite{diekmann1978thresholds,ducasse2022threshold,thieme1977model} for some results concerning the existence of traveling waves.\\

The main question concerning system \eqref{SIR original} is wether or not it predicts that a disease would propagate, leading to an epidemic. In order to state this rigorously, let us define what ``propagation" means in this context.\\

We say that we have propagation of the epidemic if, for every value of the age variable $x$, there is a strictly positive proportion of the population that gets infected, {\em no matter how small} $I_0$ is, that is, we have propagation of the epidemic if and only if there is $\eta>0$ such that, for every $I_0$ continuous, non-negative, non-zero, we have, denoting $S_\infty(x) := \lim_{t\to + \infty} S(t,x)$ (this represents the number of individuals who eventually survive the disease), 
$$
\frac{S_\infty(x)}{S_0(x)} \leq 1 - \eta, \quad \forall x \in \X. 
$$
On the other hand, we say that the disease does not spread if and only, provided the initial number of infectious is small
, the proportion of contamination is small. More precisely, for every $\eta>0$, there is $\e>0$ such that, for all $I_0$ such that $\|I_0\|_{L^\infty}\leq \e$, we have
$$
\frac{S_\infty(x)}{S_0(x)}\geq 1-\eta,\quad \forall x\in \X.
$$
The main result on the SIR system \eqref{SIR original} is the ``threshold effect", which states that there is a quantity $\lambda_1 \in \R$, depending on the parameters of the system $S_0,\beta,\mu$, whose sign determines wether or not we have {\em propagation} of the disease or not.

\begin{thrm}\label{th classical}
Let $S_0,\beta,\mu$ be fixed, continuous and positive on $\X$.

Let $\lambda_1$ be the principal eigenvalue of the operator $L \in \mathcal{L}(C^0(\X))$ defined as
$$
L \ : \ \phi  \mapsto \phi - \int_y \frac{\beta(\cdot,y)S_0(y)}{\mu(y)}\phi(y)\di y.
$$
Then, the sign of $\lambda_1$ determines wether or not the epidemic spreads:
\begin{itemize}
    \item If $\lambda_1 \geq 0$, the epidemic does not spread.  
    \item If $\lambda_1<0$, the disease spreads, there is an epidemic.
\end{itemize}
\end{thrm}
Let us recall some facts concerning the principal eigenvalue of such operator $L$, that we shall use later.

The principal eigenvalue of $L$ is the unique real number $\lambda_1$ such that there is $\phi_1 \in C^0(\X)$, with $\phi_1 > 0$, such that $L\phi_1 = \lambda_1 \phi_1$. The function $\phi_1$ is called a principal eigenfunction.

The existence of the principal eigenvalue is a consequence of the Krein-Rutman theorem \cite{KreinLinear48}, which also states that its multiplicity is one, and that it is the only eigenvalue associated with a positive eigenfunction: for every other eigenvalue $\lambda$ of $L$, if there is an eigenfunction $\phi$ such that $L\phi = \lambda \phi$ with $\phi >0$, then $\lambda = \lambda_1$. \\

Let us come back to the SIR model \eqref{SIR original} and to Theorem \ref{th classical}. An important corollary is that, in the homogeneous constant case, that is, when $\beta, S_0,\mu$ are constant functions, the principal eigenvalue of $L$ is explicit: it is $1- \frac{\beta S_0}{\mu}$. Therefore, the epidemic spreads if and only if
$$
1< \frac{\beta S_0}{\mu}.
$$
In this setting, the quantity $\frac{\beta S_0}{\mu}$ is called the {\em basic reproduction number}.\\

Let us mention briefly other works related to the SIR model of Kermack and McKendrick. Although we consider non-local interactions here, another classical way to build spatial models is to add diffusion terms in the equations: this leads to reaction-diffusion systems. The long-time behavior of such models was widely studied, see \cite{DucasseLaborde, ducrotgiletti, MurrayMathematical03a} and the references therein. The introduction of additional structure (such as age, phenotypical traits, etc) makes the study of the long-time behavior more intricate, we refer to \cite{DucasseNordmann, DucasseTreton} for some details in this direction.

We also mention that SIR models can be studied starting from individual-based models. This approach, based on a formalism from kinetic theory, proved useful to model different heterogeneities within the populations, see \cite{dimarco2021kinetic, tea23, martalo2026}.

Other optimal control problems in age-structured epidemic models were investigated. In most cases, vaccination is interpreted just like we do here, that is, as removing a portion of susceptible individuals, we refer to the works \cite{castillochavez1996, feichtinger2003, fister2016, muller2000} and the references therein.

\subsection{Optimizing the vaccination strategy}

Let us come back to our model \eqref{syst}. The question we want to answer is: {\bf what is the optimal vaccination strategy that would minimize the number of deaths in the population}, under the constraint that the number of vaccines available is limited. We call this problem the {\bf Optimal vaccination problem}. \\

{\bf The Optimal Vaccination Problem.}\\

Let $(S_0,I_0)$ be chosen and consider the solution $(S,I)$ of \eqref{syst} arising from the initial datum $(S_0,I_0)$. Let us denote $S_\infty(x) \coloneqq \lim_{t\to +\infty} S(t,x)$ (this quantity is well defined as we shall see). The quantity $\int_\X S_\infty(x)\di x$ represents the number of individuals that are still susceptible after the epidemic - these individuals where never contaminated, and they where never vaccinated neither.

On the other hand, the quantity $\int_{\X}\int_0^\infty \nu(t,x)\di t \di x$ represents the number of individuals that have been vaccinated (hence, they were never contaminated neither).

In the modeling we consider, all the individuals that are contaminated eventually die. Therefore, for a given choice of $\nu$, the number of surviving individuals is given by
$$
\mathcal N(\nu)\coloneqq\int_{\mathcal X} S_{\infty}(x)\,\di x + \int_{\mathcal X}\int_0^\infty \nu(t,x)\,\di t\,\di x.
$$
The main question that motivates our study is to maximize the function $\mathcal{N}$ on the set of admissible controls
$$
\V \coloneqq \{ \nu  \in L^1_t C^0_x\cap C^0_{t,x}, \ \nu\geq 0 \ : \ S(t,x) \geq 0,\quad \forall t>0,\ x \in \X\},
$$
under appropriate constraints on the number of vaccines. It is natural to consider the set $\V$ since we do not want the quantity $S(t,x)$ to become negative.\\

Then, assuming that we have a total quantity $K \in \R^+$ of vaccines, the Optimal Vaccination Problem is the following: 
\begin{equation}\label{ocp}
 \text{(OVP)  \quad \quad} \boxed{\text{  \text{Find\ } } 
\displaystyle{\sup_{\nu \in \V } \mathcal N(\nu)}\ \text{under the constraint\quad }\  \displaystyle\int_\X \int_0^{+\infty}\nu(t,x)\di t\di x \leq K.}
\end{equation}
As we mentioned in section \ref{sec21}, we will prove that this is equivalent to the following optimization problem.\\

{\bf The Initial Vaccination Problem.}\\

\noindent Assume that the initial density of susceptible and infectious individuals $(S_0,I_0)$ are given and assume that, {\em before the epidemic starts}, we can vaccinate a number $v(x)$ of individuals aged $x$. Therefore, the quantity of vaccines used is $\int_\X v(x)\di x$. We assume that $v(x) \leq S_0(x)$ for all $x \in \X$ (so that we do not vaccinate more individuals than there are).\\

Following the same modeling hypothesis as before, the number of susceptible and infectious individuals, that we denote here $(S^\star(t,x),I^\star(t,x))$, is solution of
\begin{equation}\label{syst 0}
\left\{
\begin{array}{rll}
     \partial_t S^\star(t,x) &= -S^\star(t,x)\int_{\X} \beta(x,y)\,I^\star(t,y)\,\di y ,\quad  &t>0,\ x\in \X, \\
     \partial_t I^\star(t,x) &= S^\star(t,x)\int_{\X} \beta(x,y)\,I^\star(t,y)\,\di y - \mu(x)I^\star(t,x),\quad &t>0,\ x\in \X,
\end{array}
\right.
\end{equation}
with initial datum
$$
S^\star(0,x)\coloneqq S_0(x) - v(x),\quad I^\star(0,x)\coloneqq I_0(x).
$$
This is the same system as \eqref{syst}, but with $\nu \equiv 0$ and with initial susceptible population $S_0 -v$: there is no vaccination during the course of the epidemic, the vaccination happens {\em before}, and is taken into account in the initial datum for $S_0$.\\

The set of admissible controls for this problem will be
\begin{equation*} 
 \V^\star \coloneqq \{v \in L^1(\X) \ : \ 0\leq v\leq S_0 \}.
\end{equation*}
We measure the impact of the epidemic by considering
$$
\mathcal{N}^\star(v) = \int_\X  S^\star_\infty(x)\di x + \int_\X v(x)\di x,
$$
 where $S^\star_\infty(x) = \lim_{t\to +\infty}S^\star(t,x)$. This measures the number of surviving individuals. \\

The {\em initial vaccination problem} consists in finding the optimal vaccination strategy provided we vaccinate the individuals before the epidemic starts. In other terms, we want to find
\begin{equation}\label{ocp:equivalent}
 \text{(IVP)  \quad \quad} \boxed{\text{  \text{Find\ } } 
\displaystyle{\sup_{v \in \V^\star } \mathcal N^\star(v)}\ \text{under the constraint\quad }\  \displaystyle\int_\X v(x)\di x \leq K.}
\end{equation}
We shall prove that the two optimization problems \eqref{ocp} and \eqref{ocp:equivalent} are actually equivalent (in a sense to be made precise below), this is the object of Theorem \ref{th new opt}.\\

In some sense, the equivalence between these two problem implies that, provided the vaccines are available from the start of the epidemic, {\em an optimal strategy is to use them as soon as possible}. Although this may seem obvious {\em a priori}, the proof is somewhat involved - and it actually heavily relies on the specific shape of the equations in the SIR model. See Remark \ref{req dur} where this is discussed.

Moreover, using the equivalence between the (OVP) \eqref{ocp} and the (IVP) \eqref{ocp:equivalent}, we will be able to find optimal strategies in some specific cases. This is the object of Theorem \ref{th quali}.

\subsection{Hypotheses and main results}

In the sequel, $\X = [0,A]$ is a given interval of $\R$. We shall always assume that the following hold true:

\paragraph{Assumptions. }

\begin{enumerate} [label=\textbf{A.\arabic*},ref=A.\arabic*]
\item \label{(A1)}$S_0(x), I_0(x) \in C^{0} (\mathcal X)$, $S_0(x)>0, I_0(x)\geq 0$ for all $x\in \X$, and $I_0\not\equiv 0$.
\item \label{(A2)} $\beta \in C^0(\X\times\X)$ and $\beta(x,y) > 0$ for every $x,y \in \mathcal X$;
\item \label{(A3)} $\mu \in C^0(\X)$ and $\mu(x) > 0$ for every $x \in \mathcal X$;
\item \label{(A4)} $\nu \in \V$.
\end{enumerate}
We are now in position to state our main results. \\

We start with a technical result, that shows that the model \eqref{syst} is well-posed and that the solutions converge to a limiting state when $t$ goes to $+\infty$.\\
\begin{thrm}\label{th existence and convergence}
For $\nu \in \V$, the system \eqref{syst} has a unique solution $(S,I)$ with $S,I\geq 0$ and $S,I \in C^1(\R^+, C^0(\X))$.

Moreover, there is $S_\infty \in L^\infty(\X)$ such that, for all $x\in \X$,
    $$
    S(t,x) \underset{t\to +\infty}{\longrightarrow} S_\infty(x),\quad I(t,x)\underset{t\to +\infty}{\longrightarrow} 0.
    $$
In addition, for all $x\in \X$ where $S_\infty(x)>0$, we have
    \begin{multline}\label{lim S}
    S_\infty(x) = S_0(x)\exp\Bigg(\int_{\mathcal X}\frac{\beta(x,y)}{\mu(y)}\,S_\infty(y)\,\di y +\int_{\mathcal X}\frac{\beta(x,y)}{\mu(y)}\nu_\infty(y)\,\di y \\ -\int_0^\infty \frac{\nu(t,x)}{S(t,x)}\,\di t -\int_{\mathcal X}\frac{\beta(x,y)}{\mu(y)}\,(I_0+S_0)(y)\, \di y\Bigg),
    \end{multline}
    where
    \begin{equation}\label{def:vinfty}
    \nu_\infty(x)\coloneqq\int_{0}^\infty \nu(t,x)\,\di t.
    \end{equation}
\end{thrm}
We now turn to the optimization problem \eqref{ocp}. Our next result tells us that the two optimization problems introduced before, the (OPV) \eqref{ocp} and the (IVP) \eqref{ocp:equivalent} are equivalent.

\begin{thrm}\label{th new opt}
    We have
    $$
    \sup_{\nu \in \V \ : \ \int_t\int_x \nu \leq K}\mathcal{N}(\nu) = \sup_{v \in  \V^\star \ : \ \int_{\X} v \leq K} \mathcal{N}^\star(v).
    $$
\end{thrm}
We shall actually prove something more precise : for each vaccination strategy $\nu(t,x)$, it is always better to vaccinate each age class at the start of the epidemic. More precisely, for $\nu(t,x)$ given, if we denote $\nu_\infty(x)= \int_0^\infty\nu(\tau,x)\di \tau$ the quantity of individuals aged $x$ that will be eventually vaccinated, then, we can find a sequence of vaccination plans $\nu_\e(t,x)$, supported in $[0,\e]\times \X$ with $\int_0^\infty \nu_\e(t,x)\di \tau = \nu_\infty(x)$ (that is, for each $x$, we vaccinate the same number of individuals of age $x$ in all strategies), such that $\mathcal{N}(\nu) \leq \mathcal{N}(\nu_\e) \underset{\e\to 0}{\longrightarrow} \mathcal{N}^\star(\nu_\infty)$.\\

\begin{rmrk}\label{req dur}
    As we mentioned before, one can interpret Theorem \ref{th new opt} as saying that it is always better to vaccinate the individuals as soon as possible (provided the vaccines are available from the start). Although this might seem obvious, the proof is rather technical, and relies heavily on the specific shape of the equations in the SIR model. We even believe that it is possible to find variations of the SIR model such that the claim that ``it is best to use all vaccines as soon as possible" would be false.\\

    To understand this, consider the model \eqref{syst} where we remove the age structure, so that it becomes
\begin{equation}\label{sir ode}
    \left\{
\begin{array}{rll}
     \dot S(t) &= -\beta S(t)I(t) - \nu(t), &\quad t>0, \\
     \dot I(t) &= S(t)I(t) - \mu I(t), &\quad t>0,
\end{array}
\right.
\end{equation}
where $\beta,\mu$ are positive constants. This is an ODE system and the control $\nu$ depends only on $t$.

The evolution of this system can be seen in the phase space $(S,I)$, represented below.

\begin{center}
\includegraphics[scale=0.5]{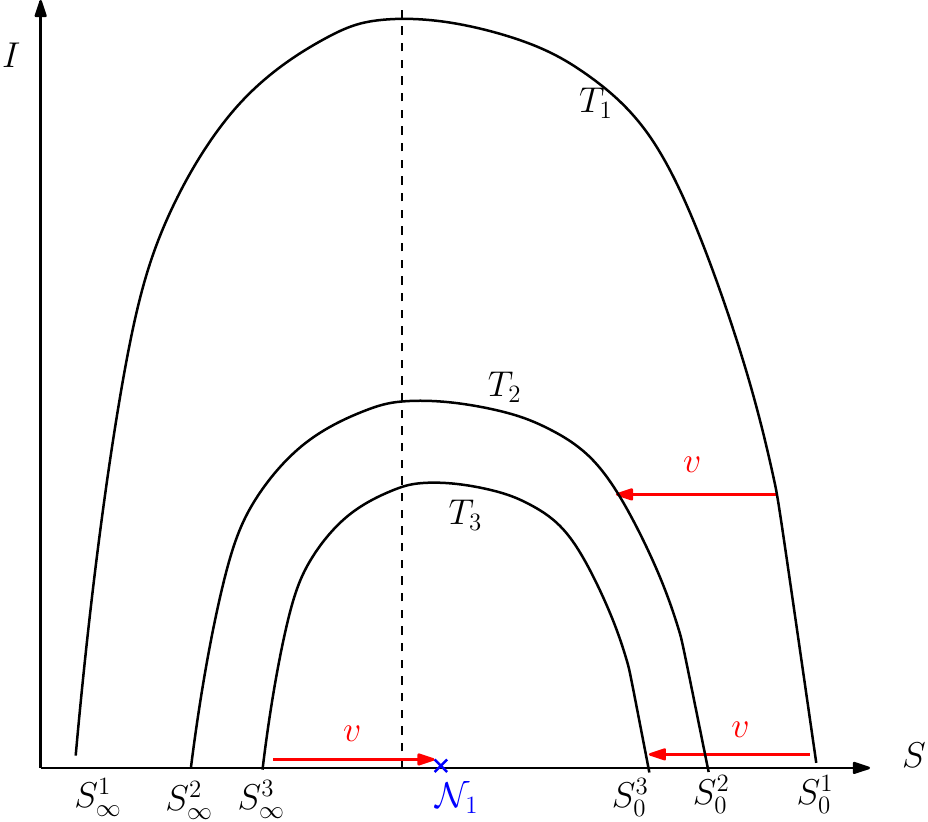}
\end{center}

We have represented three trajectories of \eqref{sir ode}, $T_1,T_2,T_3$, starting from the points $(S_0^1,I_0^1)$, $(S_0^2,I_0^2)$, $(S_0^3,I_0^3)$ respectively, with $\nu\equiv 0$ and $I_0^1, I_0^2,I_0^3 =\epsilon \ll1 $.

The effect of the vaccination term $\nu(t)$, is to translate toward the left, jumping from one trajectory to an other.

Assume that we start with a population of susceptible equal to $S_0^1$ and initial number of infectious $I_0^1$ very small. Then, if $\nu(t)\equiv 0$, the trajectory of the point $(S,I)$ would follow the trajectory $T_1$ and at the end we would have $\lim_{t\to+\infty} S(t) = S_\infty^1$ surviving individuals.

Let us compare two scenarios, one where we vaccinate very early a number $v>0$ of individuals, one where we vaccinate the same number of susceptible later.\\

We start with $S_0^1$ susceptible individuals in the initial population. Suppose that, at time $t=0$, we vaccinate $v$ individuals. Assume that $S_0^1-v=S_0^3$. Then, the dynamics would follow (after the vaccination) the trajectory $T_3$. The number of surviving individuals would be $S_\infty^3 + v$ (because $S_\infty^3$ is the number of individuals who stay susceptible after the epidemic, they were never contaminated, and $v$ is the number of individuals that we vaccinated at the start), which is represented by the point $\mathcal{N}_1$ in the image.

Consider now a different scenario, where we wait before vaccinating the susceptible. The system would follow the trajectory $T_1$, then at some strictly positive time we vaccinate a number $v$ of susceptible, so that we jump to the trajectory $T_2$ say. The number of surviving individuals in this scenario would be $S_\infty^2+v$.

We see on the image that $S_\infty^2<S_\infty^3$, so that the number of surviving individuals is larger when the vaccination occurs earlier. However, this works because, in the scenario where we vaccinate later, we jump from the trajectory $T_1$ to $T_2$, and $T_2$ is above $T_3$ (where we would jump when vaccinating early). This works well in the SIR system because the vector field of the system has a specific shape. We believe that one could imagine vector fields (and then ODE systems) for which this would not work.
\end{rmrk}


Using Theorem \ref{th new opt}, we can obtain some qualitative results on our optimization problems. We illustrate this by solving the (IVP) \eqref{ocp:equivalent} (and therefore the (OVP) \eqref{ocp}) in the specific case where the contamination function $\beta(x,y)$ is independent on its first variable (that is, $\beta(x,y)=\beta(y)$ - this means that every susceptible, no matter their age, have the same probability to get infected - however, not all infectious contribute to the same extent to the contamination), and when the quantity of vaccines available $K$ is ``not too large". We denote in the sequel $\mathds{1}_A$ the indicator function of the set $A$, that is, the function equal to $1$ on $A$ and zero everywhere else.
\begin{thrm}\label{th quali}
Let $S_0,I_0,\beta,\mu$ be fixed. Assume that $\beta$ depends only on its second variable, that is $\beta(x,y)=\beta(y)$.

    Consider the (IVP) \eqref{ocp:equivalent} with $K$ such that
    $$
    K< \min\left\{\frac{\mu}{\beta}\right\}\left( \int_\X\frac{\beta}{\mu}S_0 \right).
    $$
    We assume in addition that, for all $s>0$, the Lebesgue measure of the set $\{\frac{\beta}{\mu}=s\}$ is zero.

    Then, the function 
    $$
    v_K := S_0\mathds{1}_{\frac{\beta}{\mu}>s_K},
    $$
    where $s_K$ is such that $\int_{\frac{\beta}{\mu}>s_K}S_0 = K$,
    is a maximizer of \eqref{ocp:equivalent}.
\end{thrm}

\begin{rmrk}\label{req quali}
Theorem \ref{th quali} tells us that (up to some hypotheses on the parameters), the optimal strategy in the (IVP) \eqref{ocp:equivalent} is to vaccinate the individuals in the class ages for which the quantity $\frac{\beta}{\mu}$ is large, and to use all the vaccines available on them.

Therefore, owing to Theorem \ref{th new opt} (and to the remarks below), an optimal strategy in the (OVP) \eqref{ocp} would be to vaccinate  the individuals in the age classes for which the quantity $\frac{\beta}{\mu}$ is large and to do it ``as soon as possible". Observe that for the (OVP) \eqref{ocp}, this strategy does not belong to the admissible controls $\V$ we consider - but it will be possible to approximate it by a maximizing sequence. 
\end{rmrk}

\begin{rmrk}
    We mention that the technical hypothesis that the sets $\{\frac{\beta}{\mu}=s\}$ have zero Lebesgue measure for each $s>0$ can be removed, we will mention how in the proofs. It is nevertheless convenient to simplify some computations.

    We also believe that the restriction hypothesis on $K$ can be removed, although this would require more computations.
\end{rmrk}

We conclude this section with a simulation that illustrates some of our results. We consider the situation where $\X = [0,1]$, and where 
$$
\beta(x,y)=b\,e^{\frac{-(x-y)^2}{\sigma_{\beta}}}, \quad \mu(x)=m_\mu\,x+q_\mu.
$$
We take $b=0.05$ and $\sigma_\beta=0.05$ and $m_{\mu}=0.4$ and $q_\mu=0.1$ in the simulation.

The initial population $S_0$ is taken as
 $$
S_0(x)=\frac{1}{\sigma \sqrt{2\pi}}\,e^{-\frac{1}{2}\bigl(\frac{x-\overline x}{\sigma}\bigr)^2},
$$
with $\bar x=0.3$, $\sigma=0.5$, and $I_0$ is a small constant.

In Figure $1$, image (a) represents an optimal vaccination plan. As expected, it is concentrated near $t=0$. Images (b) and (c) show the evolution of the $S$ and $I$ populations (the time is on the horizontal axis, while the age is on the vertical axis). One can see that the optimal vaccination has completely removed all the susceptible located in some area, thus giving rise to the horizontal zone where the is no population during the evolution. The numerical result is obtained by solving the two optimization problems with the primal-dual interior point algorithm, and using the Euler scheme as integration method.

\begin{figure}[H]
\centering
\subcaptionbox{\label{Control_gau3_bsym_u2}}
{\includegraphics[trim={1.8cm 7cm 1.8cm 7cm},clip,width=5.2cm]{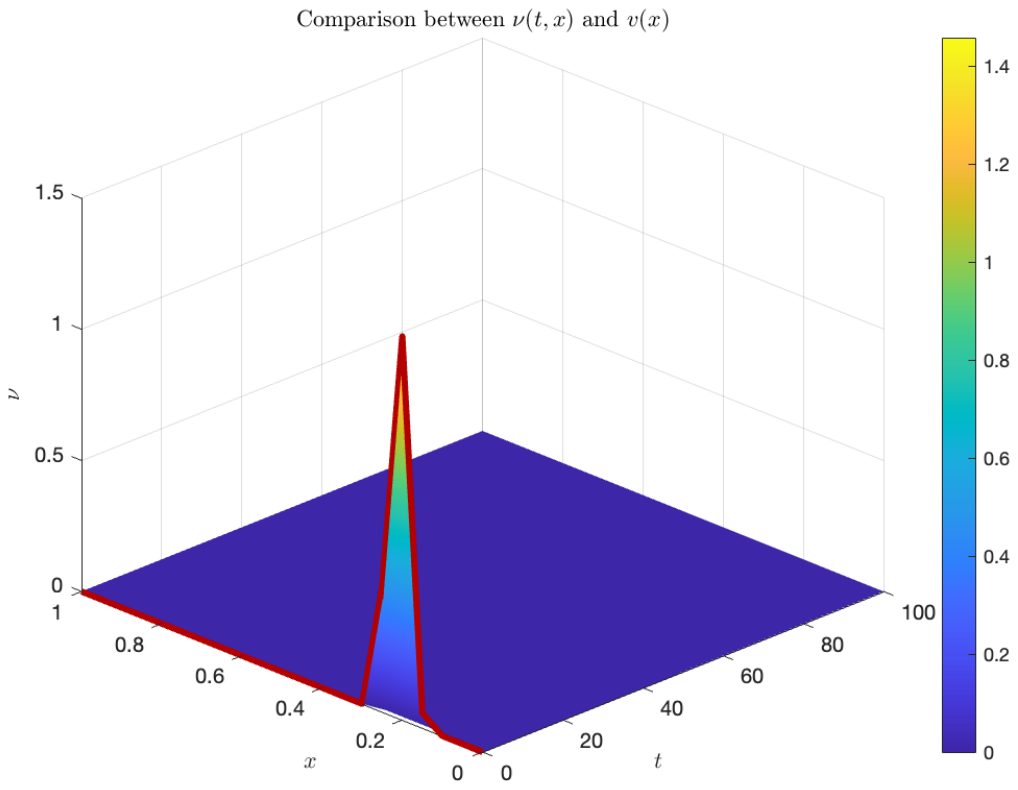}}
\subcaptionbox{\label{s_gau3_bsym_u2}}
{\includegraphics[trim={1.8cm 7cm 1.8cm 7cm},clip,width=5.2cm]{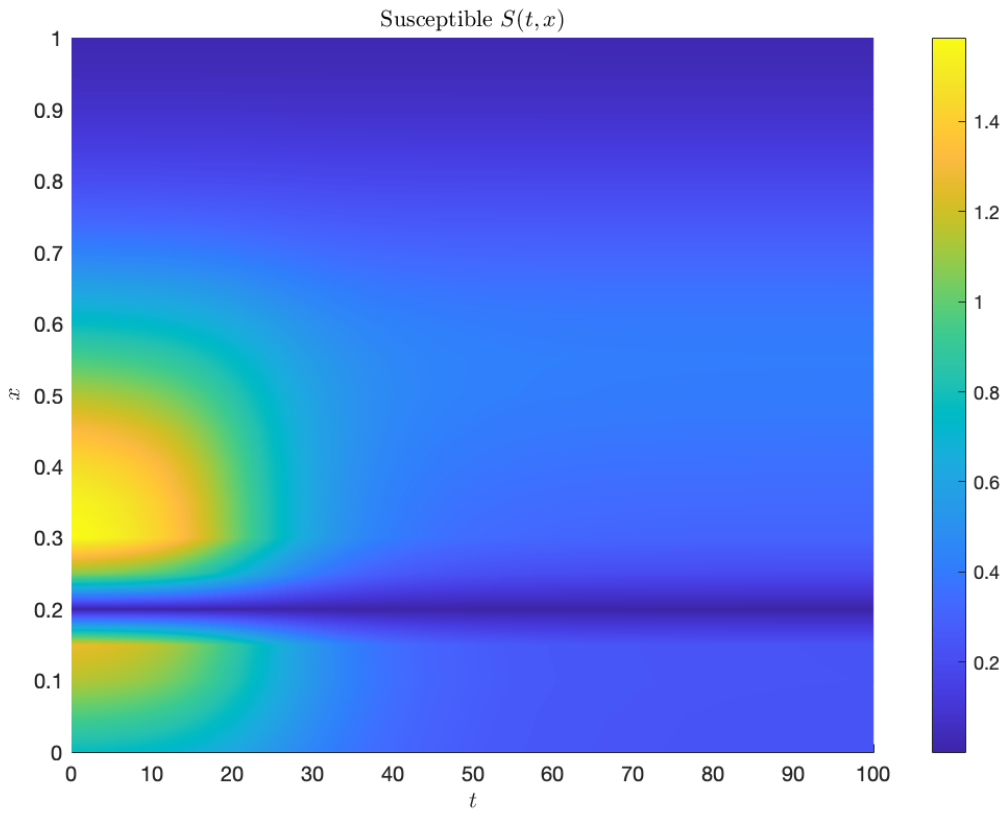}}
\subcaptionbox{\label{i_gau3_bsym_u2}}
{\includegraphics[trim={1.8cm 7cm 1.8cm 7cm},clip,width=5.2cm]{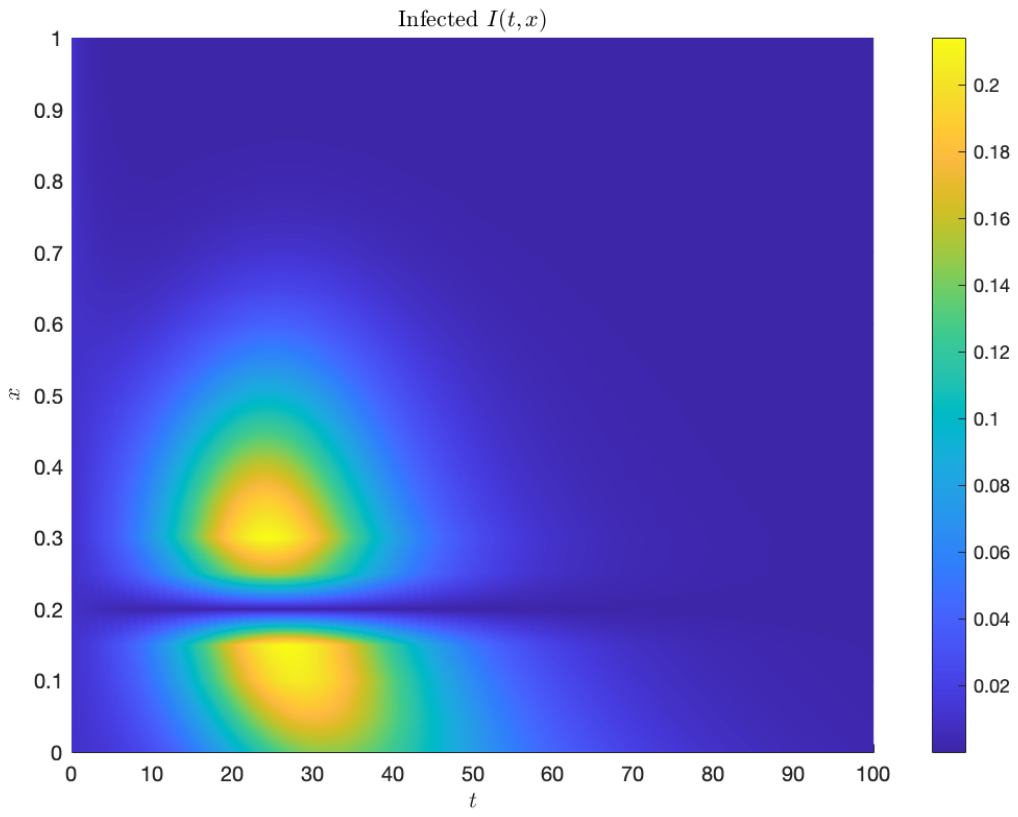}}
\caption{Optimal vaccination (a), susceptible population (b) and infected population (c).
The solid red line represents the optimal-control trajectory for the equivalent problem \eqref{ocp:equivalent}.}
\end{figure}

\section{Proofs of the analytical results}

\subsection{Existence of solutions to \eqref{syst}}

Let us now present the results concerning the existence of solutions to \eqref{syst}. The study of the long-time behavior of the solutions is considered in section \ref{sec ltb}.

\begin{prop}\label{prop existence}
    Let $\nu \in \V$ and $S_0,I_0 \in C^0(\X)$ with $S_0,I_0\geq 0$. Then, there is $(S,I) \in C^1(\R^+,C^0(\X))$, $S,I\geq 0$, solution to \eqref{syst}.
\end{prop}
The proof follows very classical arguments (rewriting the system as a fixed point for a well chosen function). Nevertheless we present it as, in the course of the proof, we provide some estimates that will be useful later.
\begin{proof} [Proof of Proposition \ref{prop existence}]
Let  $N \coloneqq \|S_0 + I_0\|_{L^\infty}$ (in particular, we have $\|S_0\|_{L^\infty}\leq N$ and $\|I_0\|_{L^\infty}\leq N$, owing to the non-negativity of $S_0,I_0$). For $T>0$ to be chosen later, we define
$$
E_T \coloneqq \{ (S,I) \in C^0([0,T],C^0(\X)) \ : \ \|S\|_{L^\infty([0,T]\times \X)} \leq 2N,\|I\|_{L^\infty([0,T]\times \X)} \leq 2N  \}.
$$
We also introduce the function $F$
\begin{multline*}
F(S,I)(t,x) = \Bigg(S_0(x) - \int_0^t S(\tau,x) \int_\X\beta(x,y) I(\tau,y)\di \tau \di y - \int_0^t\nu(\tau,x)\di \tau ,\\ I_0(x) + \int_0^t S(\tau,x) \int_\X\beta(x,y) I(\tau,y)\di \tau \di y - \int_0^t  \mu(x) I(\tau,x)\di \tau\Bigg). 
\end{multline*}
From now on, our goal is to prove that the function $F$ admits a fixed-point on $E_T$ for $T$ small enough, using Banach fixed-point theorem. This will follow from classical arguments. However, the function $F$ is not globally Lipschitz, therefore, in order to prove the existence of solutions for all time, we shall establish some estimates to be able to conclude.

\medskip
{\em Step 1. Existence of a local solution.}

First, observe that
$$
\left\|S_0 - \int_0^t S (\beta I) - \int_0^t\nu \right\|_{L^\infty([0,T]\times \X)} \leq N + T\|\beta\|_{L^\infty} (2N)^2 + T \|\nu\|_{L^\infty},
$$
and 
$$
\left\|I_0 + \int_0^t S (\beta I) - \int_0^t  \mu I \right\|_{L^\infty([0,T]\times \X)} \leq N + T\|\beta\|_{L^\infty} (2N)^2 + T \|\mu\|_{L^\infty}(2N).
$$
Taking $T>0$ small enough depending only on $N,\beta,\nu,\mu$ but not $S_0,I_0$, we ensure that $F(E_T) \subset E_T$.

Now, take $(S_1,I_1), (S_2,I_2) \in E_T$. An easy computation shows that there is $C>0$, that depends only on the coefficients $ \beta,\mu$ (we use the usual norm on the product space), such that
$$
\|F(S_1,I_1) - F(S_2,I_2) \|_{L^\infty([0,T]\times \X)} \leq C N T \|(S_1,I_I) - (S_2,I_2)\|_{L^\infty([0,T]\times \X)}.
$$
Taking $T>0$ small enough, this is a contraction. Therefore, for small $T>0$ there is a solution of the system \eqref{syst} for $t\in [0,T]$.

\medskip
{\em Step 2. Extension of the solution to all $t>0$.}

Let $(S,I)$ be the solution obtained in the first step for $t\in [0,T]$. Owing to the choice of $\nu$, we necessarily have $S\geq 0$. Therefore, we also have $I \geq 0$, and then $S+I\geq 0$ (on $[0,T]\times \X)$. Summing the two equations of \eqref{syst}, we have $\partial_t (S+I) = -\nu - \mu I \leq 0$, so that $0\leq S+I\leq S_0 + I_0$. Hence,
$$
S(T,\cdot) + I(T,\cdot) \leq S_0 + I_0,
$$
that is, $(S(T,\cdot),I(T,\cdot))$ satisfies the same properties as $S_0,I_0$. Therefore, using the same arguments as in the first step, we can extend the solution for $t\in\R^+$.
\end{proof}

\subsection{Long-time behavior of the solutions of \eqref{syst}}\label{sec ltb}

We now investigate the limiting state $S_\infty(x) = \lim_{t\to+\infty}S(t,x)$ which appears in the formulation of our optimal control problem \eqref{ocp}. The existence of this limit state is direct: for an admissible strategy $ \nu \in \V$, the function $S(t,x)$, solution of \eqref{syst}, is non-negative and is time non-increasing. Therefore, the limiting state $S_\infty$ is well defined, the convergence of $S$ toward $S_\infty$ is pointwise, and $S_\infty$ is also non-negative.\\

The next proposition proves the relation \eqref{lim S} and concludes the proof of Theorem \ref{th existence and convergence}.
\begin{prop}\label{th lim S} Let $(S,I)$ be the solution of \eqref{syst} arising from the initial datum $(S_0,I_0)$ with vaccination term $\nu\in \V$ and recall that $\nu_\infty(x) := \int_0^\infty\nu(\tau,x)\di \tau$. We have 
$$
S(t,x) \underset{t\to+\infty}{\longrightarrow} S_\infty(x), \quad I(t,x) \underset{t\to+\infty}{\longrightarrow}0.
$$
These convergences hold pointwise in $x$. 
    
    For all $x$ such that $S_\infty(x)>0$, \eqref{lim S} holds true, that is
        \begin{multline*}
    S_\infty(x) = S_0(x)\exp\Bigg(\int_{\mathcal X}\frac{\beta(x,y)}{\mu(y)}\,S_\infty(y)\,\di y +\int_{\mathcal X}\frac{\beta(x,y)}{\mu(y)}\nu_\infty(y)\,\di y \\ -\int_0^\infty \frac{\nu(\tau,x)}{S(\tau,x)}\,\di \tau -\int_{\mathcal X}\frac{\beta(x,y)}{\mu(y)}\,(I_0+S_0)(y)\, \di y\Bigg).
    \end{multline*}
\end{prop}
Before proving Proposition \ref{th lim S}, we start with a simple observation.
\begin{lem}
     If $x\in \X$ is such that $S_\infty(x) >0$, we have
     $$
     \int_0^\infty \frac{\nu(\tau,x)}{S(\tau,x)}\,\di \tau<+\infty.
     $$
\end{lem}
\begin{proof}
This comes directly from the observation that, if $S_\infty(x)>0$, using that $S$ is time non-increasing, we have $\int_0^\infty \frac{\nu(\tau,x)}{S(\tau,x)}\di \tau \leq \frac{1}{S_\infty(x)}\int_0^\infty \nu(\tau,x)\di \tau$. This lemma tells us that the formula \eqref{lim S} indeed makes sense where $S_\infty>0$.
\end{proof}
The proof of Proposition \ref{th lim S} relies on the following
\begin{lem} Let $(S,I)$ be solution of \eqref{syst}. For all $(t,x) \in \R^+\times \X$ such that $S(t,x)>0$, we have
\begin{align}\label{eq u}
S(t,x)=S_0(x)\exp\Biggl(&\int_{\mathcal X}\int_0^t \beta(x,y)(S(t-\tau,y)-S(0,y))\,e^{-\mu(y)\tau}\di \tau \,\di y \nonumber \\
&+\int_{\mathcal X}\beta(x,y)\int_0^t \Biggl( \int_{0}^{t-\tau} \nu(s,y)\,\di s\Biggr)e^{-\mu(y)\tau} \di \tau \,\di y\nonumber \\ 
&-\int_{\mathcal X}\frac{\beta(x,y)}{\mu(y)}\,I_0(y)(1-e^{-\mu(y)t})\di y-\int_0^t \frac{\nu(\tau,x)}{S(\tau,x)}\,\di \tau \Biggr).
\end{align}

\end{lem}

\begin{proof}
    Let $(t,x) \in \R^+ \times \X$ be such that $S(t,x)>0$. Because $S$ is continuous and non increasing with respect to $t$, then $S(\tau,x)>0$ for all $\tau \in  [0,t)$.

First, using equation \eqref{syst}$_1$ we find
\beq \label{eq:dtu_v}
-\frac{\pa_t S(t,x)}{S(t,x)}=\int_{\mathcal X} \beta(x,y)\,I(t,y)\,\di y +\frac{\nu(t,x)}{S(t,x)}.
\eeq
    Next, summing the equations \eqref{syst}$_{1,2}$, we get
$$
\pa_t I+\mu\,I=-\pa_t S-\nu,
$$
which yields, integrating with respect to the $t$ variable,
$$
I(t,x)=-\int_0^t \pa_t S(t-\tau,x)\,e^{-\mu(x)\tau}\di \tau -\int_0^t \nu(t-\tau,x)\,e^{-\mu(x)\tau}\di \tau +I_0(x)\,e^{-\mu(x)t}.
$$
Now, using this expression of $I$ we compute $\int_\X \beta(x,y)I(t,y)\di y$:
\begin{align}
\int_{\mathcal X} \beta(x,y)\,I(t,y)\,\di y=&-\int_{\mathcal X}\int_0^t \beta(x,y)\,\pa_t S(t-\tau,y)\,e^{-\mu(y)\tau}\di \tau \,\di y +\nonumber\\
&-\int_{\mathcal X}\int_0^t \beta(x,y)\,\nu(t-\tau,y)\,e^{-\mu(y)\tau}\di \tau \,\di y+\nonumber\\
&+\int_{\mathcal X} \beta(x,y)\,I_0(y)\,e^{-\mu(y)t}\di y.\nonumber
\end{align}
Inserting this expression of $\int_\X \beta(x,y)I(t,y)$ in \eqref{eq:dtu_v} and integrating it with respect to $t$, we get \eqref{eq u}.
\end{proof}
We can now turn to the proof of Proposition \ref{th lim S}.

\begin{proof}[Proof of Proposition \ref{th lim S}]

Let $x$ be such that $S_\infty(x)>0$. First, owing to the monotonous convergence $S(t,x)$ toward $S_\infty(x)>0$ and using that $\inf_{y\in \X}\mu(y)>0$, we get that
    $$
     \int_0^t \frac{\nu(\tau,x)}{S(\tau,x)}\,\di \tau +\int_{\mathcal X}\beta(x,y)\,I_0(y)\,\frac{1}{\mu(y)}(1-e^{-\mu(y)t})\,\di y
\underset{t\to+\infty}{\longrightarrow}
     \int_0^\infty \frac{\nu(\tau,x)}{S(\tau,x)}\,\di \tau +\int_{\mathcal X}\frac{\beta(x,y)}{\mu(y)}\,I_0(y)\,\di y,
    $$
and the expression on the right-hand side is finite.\\

Next, to conclude the proof, it remains to show that
\begin{multline}\label{lim 2}
\int_{\mathcal X}\int_0^t \beta(x,y)(S(t-\tau,y) - S_0(y))\,e^{-\mu(y)\tau}\di \tau \,\di y +\int_{\mathcal X}\beta(x,y)\int_0^t \Biggl(\int_0^{t-\tau}\nu(s,y)\di s\Biggr)\,e^{-\mu(y)\tau}\di \tau \, \di y \\
\underset{t\to + \infty}{\longrightarrow}
\int_{\mathcal X}\frac{\beta(x,y)}{\mu(y)}\,(S_\infty(y)-S_0(y))\,\di y +\int_{\mathcal X}\frac{\beta(x,y)}{\mu(y)}\nu_\infty(y)\,\di y.
\end{multline}
This comes from the following Cesaro-like result: for any $m>0$, for any function $f \in C^0(\R^+)$ such that $f(t)\underset{t\to+\infty}{\longrightarrow} f_\infty \in \R$, we have $\int_0^t f(t-\tau)e^{-m\tau}\di \tau \underset{t \to +\infty}{\longrightarrow} \frac{f_\infty}{m}$. This implies that, for all $y\in\X$, we have the pointwise convergence
\begin{multline*}
\int_0^t \beta(x,y)(S(t-\tau,y) - S_0(y))\,e^{-\mu(y)\tau}\di \tau  +\beta(x,y)\int_0^t \Biggl(\int_0^{t-\tau}\nu(s,y)\di s\Biggr)\,e^{-\mu(y)\tau}\di \tau \\ \underset{t\to+\infty}{\longrightarrow} \frac{\beta(x,y)}{\mu(y)}\,(S_\infty(y)-S_0(y)) +\frac{\beta(x,y)}{\mu(y)}\nu_\infty(y).
\end{multline*}
Now, using the dominated convergence theorem, we can integrate with respect to $y\in\X$ to get the convergence in \eqref{lim 2}, which concludes the proof.
\end{proof}

\subsection{The optimization problem}

We now turn to the proof of Theorem \ref{th new opt}, that is, we show the equivalence between the two optimization problems \eqref{ocp} and \eqref{ocp:equivalent}.

In the present section, the initial data $S_0,I_0$ are fixed. We also consider an admissible strategy $\nu \in \V$ fixed, and we recall the definition \eqref{def:vinfty}, i.e.
$$
\nu_\infty(x) \coloneqq \int_0^\infty \nu(\tau,x)\di \tau.
$$
We recall that $(S,I)$ denotes the solutions of the system \eqref{syst} with vaccination strategy $\nu$. We also denote $(S^\star,I^\star)$ the solution of the system \eqref{syst 0} where there is no vaccination $(\nu \equiv 0)$ but with initial datum $(S_0-\nu_\infty, I_0)$.\\

We shall first prove that
$$
\mathcal{N}(\nu) \leq \mathcal{N}^\star(\nu_\infty).
$$
This is the object of section \ref{upper}.

Then, in section \ref{sequence}, we show that for any $v \in \V^\star$, there is a sequence of strategies $\nu_\e \in \V$ such that 
$$
\mathcal{N}(\nu_\e) \underset{\e\to 0}{\longrightarrow} \mathcal{N}^\star(v).
$$
We start with two easy lemmas.

\begin{lem}\label{lem v}
    Let $\nu \in \V$ be an admissible strategy for \eqref{syst} and let $\nu_\infty(x):= \int_0^\infty \nu(\tau,x)\di \tau$ $\in \V^\star$.
    
    Then, we have $0\leq \nu_\infty <S_0$, in particular, $\nu_\infty \in \V^\star$ is an admissible strategy for \eqref{syst 0}. 
\end{lem}
\begin{proof}
This is readily seen by integrating the first equation in \eqref{syst} with respect to the $t$ variable: this gives, for all $x\in\X$,
    $$
    0\leq S_\infty(x) = S_0(x) -\int_0^\infty S(\tau,x)\int_\X\beta(x,y)I(\tau,y) \di y \di \tau - \nu_\infty(x).
    $$
    Since $\int_0^\infty S(\tau,x)\int_\X\beta(x,y)I(\tau,y) \di y \di \tau > 0$, this yields the result.
    \end{proof}

\begin{lem}\label{req Sinfty} Let $v\in \V^\star$ be chosen and consider $(S^\star,I^\star)$ solution of \eqref{syst 0}. Then, $S^\star(t,x)$ converges pointwise when $t$ goes to $+\infty$ to a limit $S^\star_\infty $ which satisfies
    \begin{equation}\label{S infty opt}
     S^\star_\infty(x) = (S_0(x) - v(x))\exp\left(\int_{\mathcal X}\frac{\beta(x,y)}{\mu(y)}\,S^\star_\infty(y)\,\di y -\int_{\mathcal X}\frac{\beta(x,y)}{\mu(y)}\,(I_0(y) + S_0(y) - v(y))\, \di y\right).
    \end{equation}
In addition, if $v$ can be written as $v(x) = \int_0^\infty \nu(t,x)\di t$ for some $\nu \in \V$, then we have that $S_\infty^\star \in C^0(\X)$ and $\inf_\X S_\infty^\star >0$.
\end{lem}
\begin{proof}The expression of $S_\infty^\star$ comes directly by observing that \eqref{syst 0} is exactly \eqref{syst} but with initial datum $S_0-v$ instead of $S_0$ and with $\nu\equiv 0$, so that Lemma \ref{req Sinfty} is a consequence of Proposition \ref{th lim S}.
When we have $v(x) = \int_0^\infty\nu(t,x)\di \tau$ for some $\nu \in \V$, then we have $v \in C^0(\X)$ (because $\nu \in L^1_tC^0_x$). Hence $S_0 - v$ is continuous, and because $\beta,I_0,\mu$ are also continuous, then the continuity of $S_\infty^\star$ follows.

The fact that the infimum is strictly positive comes from Lemma \ref{lem v}, which states that $v<S_0$ (combined with the continuity of $v,S_0$).
\end{proof}

\begin{rmrk}
    We do not claim that $S_\infty^\star$ is strictly positive for all choice of $v\in \V^\star$: clearly, if for some $x_0\in \X$ we have $v(x_0)=S_0(x_0)$, then $S^\star(t,x_0)=0$ for all subsequent times $t\geq 0$ and $S_\infty^\star(x_0)=0$. This holds true for controls $v\in \V$ of the form $\int_0^\infty\nu(t,x)\di t$, where $\nu\in \V$.

    Similarly, for $v\in\V^\star$, $S_\infty^\star$ need not be continuous in general.
\end{rmrk}

\subsubsection{Upper bound on $\mathcal{N}$}\label{upper}

This section is dedicated to proving the following
\begin{prop}[Upper bound on the optimal control problem]\label{prop:tildeU}
Let $S_0,I_0,\nu$ be given, and let $\nu_\infty\coloneqq\int_0^\infty \nu(\tau,\cdot)\di \tau$.

Let $(S,I)$ be the solution of \eqref{syst} arising from the initial datum $(S_0,I_0)$ with vaccination strategy $\nu$, and let $(S^\star,I^\star)$ be the solutions of \eqref{syst 0} arising from initial datum $(S_0-\nu_\infty,I_0)$.

We denote the corresponding limit states $S_\infty = \lim_{t\to+\infty}S(t,\cdot)$ and $S^\star_\infty = \lim_{t\to+\infty}S^\star(t,\cdot)$.

Then, we have
$$
S_\infty \leq S_\infty^\star.
$$
\end{prop}
It follows in particular from this proposition that, for $\nu\in \V$ and $\nu_\infty\coloneqq\int_0^\infty\nu(\tau,\cdot) \di \tau$, we have
$$
\mathcal{N}(\nu) \leq \mathcal{N}^\star(\nu_\infty).
$$
Therefore, the supremum in \eqref{ocp} is smaller than the supremum in \eqref{ocp:equivalent}.

A key tool in the proof of Proposition \ref{prop:tildeU} will be the following comparison principle for integral equations.

\begin{prop}\label{prop comp}
Let $k \in C^0(\X\times \X)$ be such that $k>0$, let $g \in C^1(\R)$ be concave, increasing, such that $g(0)=0$ and $g\leq 1$. \\

Let $U_1,U_2$ be two non-negative, bounded from above functions, with $U_1$ lower semi-continuous and $U_2$ upper semi-continuous, that satisfy
\begin{itemize}
\item (Super/sub solutions) For all $x\in \X$, we have
$$
U_2(x) - \int_y k(x,y)g(U_2(y))\di y \leq U_1(x) - \int_y k(x,y)g(U_1(y))\di y. 
$$
    \item (Stability) The principal eigenvalue of the operator $\phi \mapsto \phi - \int_{y\in \X}k(\cdot,y)g^\prime(U_1(y))\phi(y)\di y$ is strictly positive, that is, there are $\lambda_1, \phi_1>0$ such that 
    $$
    \phi_1(x) - \int_y k(x,y) g^\prime(U_1(y))\phi_1(y)\di y = \lambda_1 \phi_1(x).
    $$
\end{itemize}
Then,
$$
U_2 \leq U_1.
$$
\end{prop}
We recall that the notion of {\em principal eigenvalue} was explained in the introduction, section~\ref{sec sir classical}. The connection between maximum/comparison principles and the sign of the principal eigenvalue is classical, at least for elliptic and parabolic operators (we refer to \cite{protter2012maximum} for more details). \\

We postpone the proof of Proposition \ref{prop comp} for later, we now show how it implies Proposition~\ref{prop:tildeU}.

\begin{proof}[Proof of Proposition \ref{prop:tildeU}]

To prove that $S_\infty \leq S_\infty^\star$, we do the following change of function: we define, for $M>0$
$$
U^\star \coloneqq -\ln\left(\frac{S^\star_\infty}{S_0}\right),\quad U_M \coloneqq \min\left\{ -\ln\left(\frac{S_\infty}{S_0}\right),M\right\}.
$$
Observe that we have $0\leq U^\star$, $U_M$ and these functions are also bounded from above. For $U_M$ the boundedness comes from its definition (taking the minimum with a constant is necessary, indeed $S_\infty$ could vanish). We also have that $U_M$ is lower semi-continuous (this comes from the fact that $S_\infty$ is upper-semi continuous, as it is the non-increasing limit when $t$ goes to $+\infty$ of the continuous functions $S(t,\cdot)$).

The function $U^\star$ is bounded and continuous due to Lemma \ref{req Sinfty}, which states that $S_\infty^\star$ is continuous and has strictly positive infimum.\\

Our goal is now to show that, up to taking $M$ large enough, we can apply Proposition \ref{prop comp} to $U_M$, $U^\star$ to get $U^\star \leq U_M$.\\

{\em Step 1. Showing that $U^\star,U_M$ are super and subsolution of a common equation.}\\
It follows from equation \eqref{syst}$_1$ that $\pa_t S \le -\nu$, hence 
$$
 S(t,x) \le S_0(x) -\int_0^t \nu(\tau,x)\,\di \tau,\quad \forall t>0,\ x\in \X.
$$
Therefore,
$$
\int_0^\infty \frac{\nu(t,x)}{S(t,x)}\,\di t \ge \int_0^{\infty} \frac{\nu(t,x)}{S_0(x)-\int_0^t\nu(\tau,x)\di \tau} \di t=\ln \Biggl(\frac{S_0(x)}{S_0(x)-\nu_{\infty}(x)}\Biggr).
$$
Now, plotting this in the expression giving $S_\infty$, \eqref{lim S}, we get
\begin{multline*}
S_\infty(x) \leq S_0(x)\exp\Big(\int_{\mathcal X}\frac{\beta(x,y)}{\mu(y)}\,S_\infty(y)\,\di y -\int_{\mathcal X}\frac{\beta(x,y)}{\mu(y)}\,(I_0 + S_0)(y)\,\di y \\ + \int_\X \frac{\beta(x,y)}{\mu(y)}\nu_\infty(y)\, \di y - \ln \left(\frac{S_0(x)}{S_0(x)- \nu_\infty (x)}\right)\Big).
\end{multline*}
On the other hand, we have, thanks to Lemma \ref{req Sinfty},
\begin{multline*}
S^\star_\infty(x) = S_0(x)\exp\Big(\int_{\mathcal X}\frac{\beta(x,y)}{\mu(y)}\,S^\star_\infty(y)\,\di y -\int_{\mathcal X}\frac{\beta(x,y)}{\mu(y)}\,(I_0+ S_0)(y)\,\di y \\ + \int_\X \frac{\beta(x,y)}{\mu(y)}\nu_\infty(y)\, \di y - \ln\left(\frac{S_0(x)}{S_0(x)- \nu_\infty (x)}\right)\Big).
\end{multline*}
This rewrites as
$$
U^\star(x) = \int k(x,y)g(U^\star(y))\di y + R(x),
$$
where
$$
k(x,y) = \frac{\beta(x,y)S_0(y)}{\mu(y)},\ R(x)= \int_{\mathcal X}\frac{\beta(x,y)}{\mu(y)}\,I_0(y)\,\di y  - \int_\X \frac{\beta (x,y)}{\mu (y)}\nu_\infty(y)\, \di y + \ln\left (\frac{S_0(x)}{S_0(x)- \nu_\infty (x)}\right),
$$
and
$$
g(x) = 1-e^{-x}.
$$
Now, for all $x\in \X$ such that $\displaystyle U_M(x) = -\ln\left (\frac{S_\infty(x)}{S_0(x)}\right)$, using $g\left(-\ln\left(\frac{S_\infty}{S_0}\right)\right)=\frac{S_0-S_\infty}{S_0}$, we find
$$
U_M(x) \geq \int_{y\in \X} k(x,y)g\left(-\ln\left(\frac{S_\infty(y)}{S_0(y)}\right)\right)\di y + R(x).
$$
Since $\displaystyle U_M(y) \leq -\ln\left (\frac{S_\infty(y)}{S_0(y)}\right)$ for all $y$, and using that $g$ is non-decreasing, we indeed have, for all $x\in\X$ such that $\displaystyle U_M(x) = -\ln\left (\frac{S_\infty(x)}{S_0(x)}\right)$,
$$
U_M(x) \geq \int k(x,y)g(U_M(y)))\di y + R(x).
$$
Now, for all $x\in \X$ such that $U_M(x)=M$, that is, $\frac{S_\infty(x)}{S_0(x)}\leq e^{-M}$, up to taking
$$
M \geq \|k\|_{L^\infty}\vert \X\vert + \|R\|_{L^\infty},
$$
we indeed have (remember that $g\leq 1$)
$$
U_M(x) \geq \int k(x,y)g(U_M(y))\di y + R(x).
$$
Therefore, the first condition in Proposition \ref{prop comp} is verified.

\medskip
{\em Step 2. The stability condition.}\\
Let us consider now the operator
$$
L_M \ : \ \phi \in C^0(\X) \mapsto \phi - \int_y k(\cdot,y)g^\prime(U_M(y))\phi(y)\di y,
$$
and we denote $\lambda_M$ its principal eigenvalue. First, we prove that, for $M$ large enough, we have $\lambda_M>0$.\\

We start by considering the following operator $L^\star$
    $$
L^\star \ : \ \phi \in C^0(\X) \mapsto \phi - \int_y  \frac{\beta(y,\cdot)S_\infty(y) }{\mu(y)}\phi(y)\di y.
$$
Let $\lambda,\phi$ be a principal eigenvalue and positive eigenfunction for $L^\star$. Let us prove that $\lambda>0$. We argue by contradiction and assume that $\lambda\leq 0$. We denote $e(t) \coloneqq \int_x\frac{1}{\mu(x)} \phi(x)I(t,x)\di x$. We know that $e(t) \to 0$ when $t$ goes to $+\infty$ due to Proposition \ref{th lim S}.

However, multiplying the equation for $I$ in \eqref{syst} by $\frac{\phi}{\mu}$ and integrating over $x$ yields
$$
\partial_t \int_x\frac{I(t,x)\phi(x)}{\mu(x)}\di x = \int_x \int_y \frac{\beta(x,y)S(t,x) \phi(x)}{\mu(x)}I(t,y)\di y \di x - \int_x I(t,x)\phi(x)\di x.
$$
Hence, 
$$
\partial_t \int_x\frac{I(t,x)\phi(x)}{\mu(x)}\di x \geq \int_x\left( \int_y  \frac{\beta(y,x)S_\infty(y) }{\mu(y)}\phi(y)\di y -  \phi(x) \right)I(t,x)\di x = -\int_x L^\star(\phi)(x)I(t,x)\di x.
$$
This rewrites
$$
e^\prime(t) \geq - \lambda \int_x I(t,x)\phi(x)\di x \geq -\lambda (\max \mu)e(t).
$$
Because we assumed that $\lambda \leq 0$, this implies that $e(t)$ is non-decreasing, hence it is bounded from below by $e(0)>0$, which is in contradiction with $e(t)\underset{t\to+\infty}{\longrightarrow} 0$. Thus, by contradiction, $\lambda>0$.\\

Next, define the operator 
$$
L \ : \ \phi \in C^0(\X) \mapsto \phi - \int_y k(x,y)g^\prime\left(-\ln\left(\frac{S_\infty(y)}{S_0(y)} \right)\right)\phi(y)\di y = \phi - \int_y \frac{\beta(x,y)}{\mu(y)}S_\infty(y)\phi(y)\di y.
$$
This operator is the adjoint of $L^\star$ in $L^2(\X,\frac{S_\infty(x)}{\mu(x)}\di x)$, hence they have the same principal eigenvalue $\lambda$, which is strictly positive as we just proved.\\

Let us now show that $\lambda_M \underset{M\to+\infty}{\longrightarrow} \lambda$. For each $M>0$, we take $\phi_M$ to be a non-negative principal eigenfunction associated with $\lambda_M$, that we normalize so that $\|\phi_M\|_{L^\infty(\X)}=1$.\\

First, let us show that $(\lambda_M)_{M>0}$ is increasing with respect to $M$. Take $M<M'$. Let us show that $\lambda_M \leq \lambda_{M'}$. We take $\e>0$ such that $\e = \sup\{\eta>0 \ : \  \phi_{M} > \eta \phi_{M'}\}$. We have that $\phi_{M} -\e\phi_{M'} \geq 0$ and there is a contact point $x_0$ such that $\phi_{M}(x_0) -\e\phi_{M'}(x_0)=0$. Because $U_M\leq U_{M'}$ (by definition) and because $g^\prime$ is non-increasing, we have, for all $x\in \X$,
$$
\phi_{M}(x) -\e\phi_{M'}(x) \geq \int_y k(x,y)g^\prime(U_{M'}(y))(\phi_{M} -\e\phi_{M'})(y)\di y + \lambda_{M} \phi_{M}(x) - \e \lambda_{M'} \phi_{M'}(x).
$$
Assume, by contradiction, that $\lambda_M' < \lambda_M$. This would yield
$$
\phi_{M}(x) -\e\phi_{M'}(x) \geq \int_y k(x,y)g^\prime(U_{M'}(y))(\phi_{M} -\e\phi_{M'})(y)\di y + \lambda_{M'} (\phi_{M}(x) - \e  \phi_{M'}(x)).
$$
Evaluating this at the contact point $x_0$ would lead to $\phi_{M}=\e\phi_{M'}$ and then to $\lambda_{M'}=\lambda_M$, which would be a contradiction, hence the result.\\

Therefore, the eigenvalue $\lambda_M$ increases with $M$. It is also easy to see that, for all $M>0$, we have $\lambda_M \leq 1$. We can then consider its limit $\lambda_\infty := \lim_{M\to+\infty}\lambda_M$. It is readily seen that the family of eigenfunction $(\phi_M)_{M>0}$ is uniformly equicontinuous (and bounded by hypothesis), so that it converges, up to extraction, uniformly as $M\to+\infty$ to a non-negative function $\phi_\infty$ that satisfies
$$
\phi_\infty - \int_y \frac{\beta(\cdot,y)}{\mu(y)}S_\infty(y)\phi_\infty(y)\di y = \lambda_\infty \phi_\infty.
$$
Moreover, by uniqueness of the principal eigenvalue (owing to the Krein-Rutman theorem), we have $\lambda_\infty = \lambda>0$.\\

Therefore, up to taking $M$ large enough, we find that $\lambda_M>0$, which concludes this proof.
\end{proof}
We now prove the comparison principle, Proposition \ref{prop comp}. The proof relies on two lemmas.
\begin{lem}[Strong maximum principle]\label{smp}
Let $p, z$ be two non-negative, bounded functions, with $z$ lower semi-continuous, such that
$$
z(x) \geq \int_{y\in\X} k(x,y)g^\prime(p(y)) z(y)\di y,\quad \forall x \in \X.
$$
Then, if there is $x_0 \in \X$ such that $z(x_0)=0$, we have $z\equiv 0$.
\end{lem}
\begin{proof} The proof is straightforward: if $z(x_0)=0$, then we have $\int k(x_0,y)g^\prime(p(y)) z(y)\di y \leq 0$. Since the quantities under the integral are all nonnegative, we find that, for all $y\in \X$, we have $k(x_0,y)g^\prime(p(y)) z(y)=0$. This implies that $z\equiv 0$, because $k(x_0,y)g^\prime(p(y))=k(x_0,y)e^{-p(y)}>0$ for all $y\in \X$.
\end{proof}

\begin{lem}[Stability implies comparison]\label{stab comp}
    Let $p$ be a non-negative, bounded function. Assume that the principal eigenvalue of the operator
    $$
    \phi\in C^0(\X) \to \phi - \int_y k(\cdot,y) g^\prime(p(y)) \phi(y)\di y 
    $$
    is strictly positive.

    Then, if $z$ is a bounded, lower semi-continuous function such that, for all $x\in\X$,
    $$
    z(x) \geq \int_y k(x,y) g^\prime(p(y))z(y)\di y,
    $$
    we have $z\geq 0$. 
\end{lem}
\begin{proof}
Let $\lambda>0$ be the principal eigenvalue of the operator and let $\phi\in C^0(\X)$ be an associated positive eigenfunction. We have
$$
\phi = \int_y k(\cdot,y)g^\prime(p(y))\phi(y)\di y +\lambda \phi.
$$
    Let $M \coloneqq \inf\left\{ m \in \R \ : \ z \geq - m\phi \right\}$. Owing to the lower semi-continuity of $z$, we have $z\geq - M\phi$ and there is $x_0\in \X$ such that $z(x_0)+ M\phi(x_0) = 0$.

    If $M\leq 0$, the result follows. Assume by contradiction that $M>0$. Then, because we have (thanks to the non-negativity of $\lambda,\phi)$
    $$
    z+M\phi \geq \int_y k(\cdot,y) g^\prime(p(y))(z(y)+M\phi(y))\di y, 
    $$
    the strong maximum principle, Lemma \ref{smp}, applied to $z+M\phi$ yields that $z+M\phi \equiv 0$. Hence,
    $$
    \phi(x) \leq \int_y k(x,y) g^\prime(p(y)) \phi(y)\di y,\quad \forall x\in \X,
    $$
    which implies that $\lambda \phi \equiv 0$, which is a contradiction. The result is proven.
\end{proof}

We now can use the previous lemmas to prove Proposition \ref{prop comp}.
\begin{proof}[Proof of Proposition \ref{prop comp}]
Let $z = U_1 - U_2$, where $U_1,U_2$ are as required in Proposition \ref{prop comp}. The function $z$ is therefore bounded and lower semi-continuous. Because $g$ is concave, we have, for all $x\in\X$,
$$
z(x) \geq \int_y k(x,y)(g(U_1(y)) - g(U_2(y)))\di y \geq  \int_y k(x,y) g^\prime(U_1(y))z(y)\di y.
$$
The result then follows from Lemma \ref{stab comp}.
\end{proof}

\subsubsection{Construction of a maximizing sequence}\label{sequence}

Proposition \ref{prop:tildeU} shows that, for any $\nu \in \V$, we can build $\nu_\infty \in \V^\star$ as $\nu_\infty(x) := \int_0^\infty \nu(\tau,x)\di \tau$ and we have $\mathcal N(\nu)\leq \mathcal N^\star(\nu_\infty)$. 

Therefore, the supremum in the optimal value problem \eqref{ocp} is lesser than the supremum in the initial value problem \eqref{ocp:equivalent}.

This section is dedicated to proving that the two supremum are actually equal. To do so, we shall prove that for any $v \in \V^\star$, there is a sequence of functions $\nu_\e \in \V$ such that $\mathcal N(\nu_\e)\underset{\e\to0}{\longrightarrow} \mathcal N^\star(v)$.
\begin{prop}[Construction of a maximizing sequence]\label{prop:v_eps}
Let $v \in \V^\star$ be fixed. Then, there is a sequence of admissible controls $(\nu_\e)_{\e>0} \in \V$ such that
$$
\mathcal{N}(\nu_\e)\underset{\e \to 0}{\longrightarrow} \mathcal N^\star(v).
$$
\end{prop}
\begin{proof}
    We start with considering $v\in \V^\star$ such that $v$ is also continuous and that $v<S_0$.

   Let $\phi : \R^+ \to \R^+$ be continuous, such that $\int_0^\infty \phi = 1$ and $\phi$ is compactly supported in $[0,1]$. For $\e>0$, we define
   $$
   \nu_\e(t,x) = \frac{1}{\e}\phi\left(\frac{t}{\e}\right)v(x). 
   $$
    This control need not be admissible for \eqref{ocp} (i.e., it is not guaranteed that $\nu_\e \in \V$) for all $\e>0$  {\em a priori}, but we shall see that it is admissible for $\e>0$ small enough.\\ 

    We call $(S_\e,I_\e)$ the solution of \eqref{syst} with the control $\nu_\e$ and initial datum $(S_0,I_0)$. We define the rescaled functions
    $$
    \tilde S_\e(t,x) := S_\e(\e t,x) \quad \text{and}\quad  \tilde I_\e(t,x) := I_\e(\e t,x).
    $$
    These functions satisfy
\begin{equation}\label{eq:sirv_rescaled}
\left\{
\begin{array}{rll}
\pa_t \tilde S_\e(t,x) &= -\e \tilde S_\e(t,x)\int_{\mathcal{X}}\beta(x,y)\,\tilde I_\e(t,y)\,\di y-\nu_1(t,x),  \quad &t>0,\ x\in \X, 
\\
\pa_t \tilde I_\e (t,x) &= \e \tilde S_\e(t,x)\int_{\mathcal{X}}\beta(x,y)\, \tilde I_\e(t,y)\,\di y -\e \mu(x)\, \tilde I_\e(t,x),  \quad &t>0,\ x\in \X,
\end{array}
\right.
\end{equation}
that is, they solve also a SIR model but with $\beta$,$\mu$ replaced by $\e\beta$, $\e \mu$, and $\nu_1$ independent of $\e$. \\

The idea of the proof, from now on, is to show that $(\t S_\e,\t I_\e)$ behaves ``similarly" to the solution of \eqref{syst 0} with initial datum $S_0-v$ when $\e>0$ is small enough.

This is done in two steps. First, we focus on \eqref{eq:sirv_rescaled} for $t\in[0,1]$, and we prove that $\t S_\e(1,\cdot) \approx S_0 - \int_0^1 \nu_1 = S_0 - v$ and $\t I_\e(1,\cdot) \approx I_0$. In a second step, we focus on the case where $t>1$: for such $t$, the term $\nu_1$ is zero. \\

To start, arguing as in the first step of the proof of Proposition \ref{prop existence} (existence of solutions), we can prove that, up to taking $\e>0$ small enough, there is a solution to \eqref{eq:sirv_rescaled} for $t \in [0,1]$ and $\|\t S_\e(t,\cdot)+\t I_\e(t,\cdot)\|_{L^\infty} \leq 2 \|S_0+I_0\|_{L^\infty}$ for $t\in[0,1]$. This solution need not be positive {\em a priori} (we will see later that this is actually the case for $\e>0$ small enough).\\

\medskip
{\em Step 1. Estimates between $t=0$ and $t=1$.}

\noindent Let us show that there is $A>0$, independent of $\e$, such that
\beq \label{approx:SI}
\sup_{x\in\X}\vert \tilde S_\e(1,x) - (S_0(x) - \nu_\infty(x)) \vert + \sup_{x\in\X}\vert \tilde I_\e(1,x) -  \tilde I_0(x)\vert < A\e.
\eeq
Let $C>0$ be such that $\vert \tilde S_\e(t,x)\vert +\vert \tilde I_\e(t,x)\vert\le C$, for every $t\in [0,1], x \in \mathcal X$. Equation \eqref{eq:sirv_rescaled}$_2$ gives
$$
| \pa_t \tilde I_\e(t,x)|\le \e \biggl(C^2 \|\beta\|_{L^\infty}\vert \X\vert+\| \mu \|_{L^\infty} C\biggr)\le \e \, \tilde C,\quad \forall t\in [0,1],\ x\in \X,
$$
for some $\tilde C>0$, hence
$$
I_0(x)-\e\,\tilde C \le \tilde I_\e(1,x) \le I_0(x)+\e\,\tilde C,\quad \forall t\in [0,1],\ x\in \X.
$$
Up to increasing $\t C$, it follows from the equation on $\t S_\e$, \eqref{eq:sirv_rescaled}$_{1}$, that
$$
|\pa_t\tilde S_\e(t,x)+\nu_1(t,x)| = \e \vert \tilde  S_e \int \beta \tilde I_\e - \mu\tilde I_\e\vert  \leq   \e \, \tilde C, \quad \forall t\in [0,1],\ x\in \X.
$$
Therefore,
$$
-\e\,\tilde C \le \tilde S_\e(1,x)-S_0(x)+\int_0^1\nu_1(t,x)\,\di t \le \e\,\tilde C, \quad \forall t\in [0,1],\ x\in \X,
$$
thus \eqref{approx:SI} holds true. \\

\medskip
{\em Step 2. Estimates for $t>1$.}

\noindent Now, we consider the SIR model \eqref{eq:sirv_rescaled} above for $t>1$. Because $\nu_1 \equiv 0$ when $t>1$, the functions $\t S_\e,\t I_\e$ satisfy
\begin{equation*}
\left\{    
\begin{array}{rll}
\pa_t \tilde S_\e(t,x) &= -\e \tilde S_\e(t,x)\int_{\mathcal{X}}\beta(x,y)\,\tilde I_\e(t,y)\,\di y,  \quad &t>1,\ x\in \X,
\\
\pa_t \tilde I_\e (t,x) &= \e \tilde S_\e(t,x)\int_{\mathcal{X}}\beta(x,y)\, \tilde I_\e(t,y)\,\di y -\e \mu(x)\, \tilde I_\e(t,x),  \quad &t>1,\ x\in \X,
\end{array}
\right.
\end{equation*}
with ``initial" datum $(\t S_\e(1,\cdot),\t I_\e(1,\cdot))$. This is the classical SIR system (without vaccination). We have that the limit $\t S^\e_\infty(x) \coloneqq\lim_{t\to+\infty} \t S_\e(t,x)$ satisfies
\begin{equation}\label{lim e}
\t S^\e_\infty(x) = \t S_\e(1,x)  \exp\left(\int_{\mathcal X}\frac{\beta(x,y)}{\mu(y)}\, \t S^\e_\infty(y)\,\di y-\int_{\mathcal X}\frac{\beta(x,y)}{\mu(y)}\,(\t I_\e(1,y)+\t S_\e(1,y))\, \di y\right).
\end{equation}
Because $\beta$ is continuous (and since $\X$ is bounded and that every quantity in \eqref{lim e} is bounded), we have that the family $(S^\e_\infty)_\e$ is equicontinuous and uniformly bounded. By the Arzel\`a-Ascoli theorem, it converges when $\e \to 0$ to some function $\sigma$. It follows from what was proven in the first step that $\t S_\e(1,\cdot) \underset{\e \to 0}{\longrightarrow} S_0 -v$ and $\t I_\e(1,\cdot) \underset{\e \to 0}{\longrightarrow} I_0$. Hence, $\sigma$ satisfies, for all $x\in \X$,
$$
\sigma(x) = (S_0(x) - v(x))  \exp\left(\int_{\mathcal X}\frac{\beta(x,y)}{\mu(y)}\,\sigma(y)\,\di y-\int_{\mathcal X}\frac{\beta(x,y)}{\mu(y)}\,(I_0(1,y)+ S_0(y) - v(y))\, \di y\right).
$$
This is the equation \eqref{S infty opt} satisfied by $S^\star_\infty$. There is a unique solution to this equation (this is a classical fact from the literature on the SIR model, we refer for instance to \cite{ducasse2022threshold}. However, we mention that it could also be obtained directly using our comparison principle Proposition \ref{prop comp}), hence $\sigma = S^\star_\infty$, that is,
$$
\lim_{\e \to 0} \left( \lim_{t\to +\infty}\tilde{S}_\e(t,x) \right) = S_\infty^\star
$$

\medskip
{\em Step 3. Conclusion}.

\noindent Recall that $S_\e$ is the solution of \eqref{syst} with control $\nu_\e$. Because $\t S_\e(t,x) = S_\e(\e t,x)$, we have $\lim_{t\to+\infty} S_\e(t,x) = \lim_{t\to+\infty} \t S_\e(t,x)$. Therefore, remembering that we called in the previous step $\lim_{t\to+\infty} \t S_\e(t,x) = \t S_\infty^\e(x)$, we have
   $$
   \mathcal N(\nu_\e) = \int_{\mathcal X} \t S_\infty^\e + \int_{\mathcal X}\int_t \nu_\e = \int_{\mathcal X} \t S_\infty^\e  + \int_{\mathcal X} v.
   $$
   Using the previous step, we then have
   $$
    \mathcal N(\nu_\e) \underset{\e \to 0}{\longrightarrow} \int_{\mathcal X}  S_\infty^\star  + \int_{\mathcal X} v = \mathcal{N}^\star(v),
   $$
   which concludes the proof for $v\in \V^\star$ such that $v$ is continuous and $v<S_0$.

   Because such control functions are dense for the $L^1$ convergence in $\V^\star$, we can conclude the proof by density (using similar arguments as above, we have that $\mathcal{N}^\star$ is continuous for the $L^1$ norm).
\end{proof}

Gathering Proposition \ref{prop:v_eps} with Proposition \ref{prop:tildeU}, we obtain Theorem \ref{th new opt}.

\subsection{Qualitative properties}

We now turn to the proof of Theorem \ref{th quali}. In the whole section, we assume that $S_0,I_0$ are fixed and that the hypotheses of Theorem \ref{th quali} hold true: $\beta$ depends only on its second variable, that is, $\beta(x,y)=\beta(y)$, we have 
    $$
    K< \min\left\{\frac{\mu}{\beta}\right\}\left( \int_\X\frac{\beta}{\mu}S_0 \right),
    $$
    and the set $\{\frac{\beta}{\mu}=s\}$ has zero Lebesgue measure, for each $s>0$.

Under these hypotheses, we shall completely solve the (IVP) \eqref{ocp:equivalent} (and then, also the (OVP) \eqref{ocp}).

In the whole section, for each $m\in [0,K]$, we denote $v_m\in \V^\star$ the function defined so that
\begin{equation}\label{vI}
v_m(x) :=S_0(x)\mathds{1}_{\frac{\beta}{\mu}>s_m},
\end{equation}
where $s_m>0$ is such that $\int_{\frac{\beta}{\mu}>s_m}S_0 = m$. Because of the hypothesis that the set $\{\frac{\beta}{\mu}=s\}$ has zero Lebesgue measure for each $s>0$, $s_m$ is well-defined.

We will prove Theorem \ref{th quali}, that is, we want to prove that, if $v_K$ is the function defined by \eqref{vI} with $m=K$, then, we have $v_K\in\V^\star$, $\int_\X v_K \leq K$ and
    $$
    \sup_{\substack{v\in\V^{\star} \\ \int_X v\leq K}}\mathcal{N}^\star(v)=\mathcal{N}^\star(v_K).
    $$
In other words, the function $v_K$ is a maximizer for the (IVP) \eqref{ocp:equivalent}.

The proof relies on several intermediate results. Before diving into the proof, we define two functions $\sigma_\infty,\sigma_0$ acting on $\V^\star$:
$$
\displaystyle\sigma_\infty(v) := \int_y \frac{\beta(y)}{\mu(y)}S_\infty^\star(y)\di y,\ \displaystyle \sigma_0(v) := \int_y \frac{\beta(y)}{\mu(y)}(S_0(y) - v(y))\di y.
$$
In the definition of $\sigma_\infty$, the function $S_\infty^\star$ is, as above, the limiting state for $S$, solution of \eqref{syst 0}, with initial datum $(S_0-v,I_0)$.

We start by rewriting the optimization problem \eqref{ocp:equivalent} into a more convenient form.
\begin{lem}\label{lem rewrite}
   We have
    \begin{equation}\label{new opt I}
    \begin{array}{rl}
    \displaystyle{\sup_{\substack{v\in \V^\star \\ \int_\X v \leq K}} \mathcal{N}^\star(v) }&=\displaystyle{ \sup_{m\in [0,K]}\sup_{\substack{v\in \V^\star \\ \int_\X v = m}} \mathcal{N}^\star(v)}\\
    &=  \displaystyle{\sup_{m\in [0,K]}\sup_{\substack{v\in \V^\star \\ \int_\X v = m}}\left\{ m + \left(\int_\X S_0 - m\right)e^{\sigma_\infty(v) - \sigma_0(v)}\eta\right\},}
    \end{array}
    \end{equation}
    where $\eta := \exp\left(-\int_y \frac{\beta(y)}{\mu(y)}I_0(y)\di y\right)$.
\end{lem}
\begin{proof}
    Let us start by observing that Lemma \eqref{req Sinfty} yields that
\begin{equation*}
     S^\star_\infty(x) = (S_0(x) - v(x))\exp\left(\int_{\mathcal X}\frac{\beta(y)}{\mu(y)}\,S^\star_\infty(y)\,\di y -\int_{\mathcal X}\frac{\beta(y)}{\mu(y)}\,(I_0(y) + S_0 (y)- v(y))\, \di y\right),
    \end{equation*}
which rewrites as
\begin{equation}\label{eq sigma}
         S^\star_\infty(x) = (S_0(x) - v(x))e^{\sigma_\infty(v) - \sigma_0(v)} \eta.
\end{equation}
    Observe that the quantity in the exponential in \eqref{eq sigma} does not depend on $x$ (but depends on $v$). Hence integrating \eqref{eq sigma} with respect to $x\in \X$, yields
    $$
    \int_\X S_\infty^\star =\left(\int_\X S_0-\int_\X v\right)e^{\sigma_\infty - \sigma_0}\eta.
    $$
    Therefore, \eqref{new opt I} holds true.
\end{proof}

In the next lemma, we give some monotonicity result for the functionals $\sigma_\infty,\sigma_0$.
\begin{lem}\label{lem sigma}
    Let $v_1,v_2 \in \V^\star$ be such  that $\sigma_0(v_1)\leq \sigma_0(v_2)$. Then, we have $\sigma_\infty(v_1)\geq \sigma_\infty(v_2)$.
\end{lem}
\begin{proof}
 Let $v\in \V^\star$ be chosen. If we multiply \eqref{eq sigma} by $\frac{\beta(x)}{\mu(x)}$ and we integrate, we find that (we omit the dependence on $v$ in $\sigma_0,\sigma_\infty$ for now)
    $$
    \sigma_\infty e^{-\sigma_\infty} = \sigma_0 e^{-\sigma_0}\eta.
    $$
    Define the function $f(x) = x e^{-x}$. This function is increasing for $x\in [0,1]$, decreasing for $x>1$, and we have
    $$
    f(\sigma_\infty)=f(\sigma_0)\eta.
    $$
    This allows to find $\sigma_\infty$ knowing $\sigma_0$. Because $\eta<1$, there are two possible solutions $x$ that solve $f(x)=f(\sigma_0)\eta$, but only one of them is smaller than $\sigma_0$ (see the image below). Indeed, as $S_\infty^\star\leq S_0-v$, we have $\sigma_\infty\leq \sigma_0$.

    \begin{center}
    \includegraphics[scale=0.7]{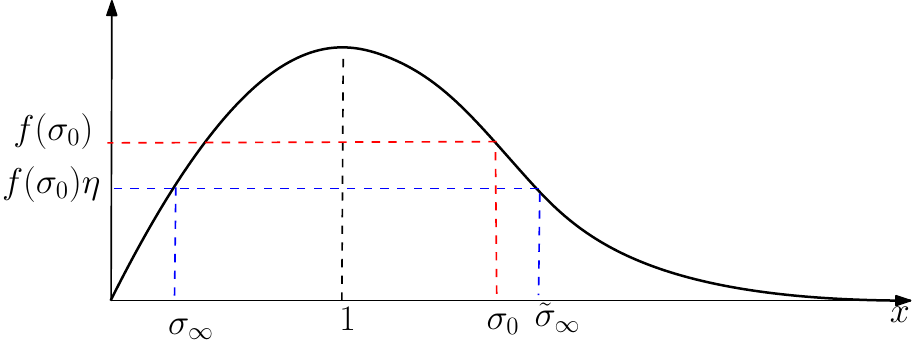}
    \end{center}
In addition, because of the hypothesis on $K$, we have $\int_y \frac{\beta}{\mu}v\leq \max\left\{\frac{\beta}{\mu}\right\} K < \int \frac{\beta}{\mu}S_0$, hence
$$
\sigma_0 = \int_y\frac{\beta(y)}{\mu(y)}(S_0(y)-v(y))\di y> 1.
$$
It is easy to see that decreasing $\sigma_0$, as long as $\sigma_0>1$, increases $\sigma_\infty$, hence the result.
\end{proof}

We are now in position to prove the following result, where we solve the optimization problem with the additional constraint that $\int_\X v = m$ is fixed.
\begin{prop}\label{prop m}
    Let $m\in [0,K]$ be fixed. Let $v_m$ be given by \eqref{vI}. We have that $v_m\in\V^\star$, $\int_\X v_m = m$ and
    $$
    \sup_{\substack{v\in\V^\star \\ \int_X v = m }}\mathcal{N}^\star(v)=\mathcal{N}^\star(v_m)
    $$
\end{prop}
\begin{proof}

{\em Step 1. A sufficient condition.}\\

Let us start with observing that it is sufficient to prove that
\begin{equation}\label{opt sigma}
\sigma_0(v_m)\leq \sigma_0(v),\quad \forall v\in\V^\star,\ \int_\X v = m.
\end{equation}
Indeed, if \eqref{opt sigma} holds true, we have, using Lemma \ref{lem sigma},
$$
\sigma_\infty(v_m)-\sigma_0(v_m)\geq \sigma_\infty(v) - \sigma_0(v),\quad \forall v\in\V^\star,\ \int_\X v = m.
$$
Therefore, thanks to Lemma \ref{lem rewrite}, and because $m\leq \int_\X S_0$, we would have (if \eqref{opt sigma} holds true)
$$
m + \left(\int_\X S_0 - m\right)e^{\sigma_\infty(v_m) - \sigma_0(v_m)}\eta\geq m + \left(\int_\X S_0 - m\right)e^{\sigma_\infty(v) - \sigma_0(v)}\eta,\quad \forall v\in\V^\star,\ \int_\X v = m,
$$
that is,
$$
\mathcal{N}^\star(v_m)\geq \mathcal{N}^\star(v),\quad \forall v\in\V^\star,\ \int_\X v = m.
$$
Let us then prove that \eqref{opt sigma} holds true.

\medskip
{\em Step 2. Optimizing $\sigma_0$.}\\

We want to prove that \eqref{opt sigma} holds true. It is sufficient to show that, for each $v\in \V^\star$ such that $\int_\X v = m$, we have
$$
\int_\X v(y)\frac{\beta(y)}{\mu(y)}\di y\leq \int_\X v_m(y)\frac{\beta(y)}{\mu(y)}\di y.
$$
It turns out that this is a consequence of the {\em bathtub principle}, see \cite{lieb2001analysis}. We reproduce the argument here for completeness. We want to prove that, for all $v\in \V^\star$ such that $\int_\X v = m$, we have
$$
\int_\X(v_m-v)\frac{\beta}{\mu}\geq 0.
$$
We have
\begin{equation*}
\begin{array}{cc}
\int_\X(v_m - v)\frac{\beta}{\mu}&=\int_{\frac{\beta}{\mu}>s_m} (v_m - v_0)\frac{\beta}{\mu}+\int_{\frac{\beta}{\mu}<s_m} (v_m - v)\frac{\beta}{\mu} \\
&=\int_{\frac{\beta}{\mu}>s_m} (S_0 - v)\frac{\beta}{\mu}-\int_{\frac{\beta}{\mu}<s_m}  v\frac{\beta}{\mu}\\
&\geq s_m\left( \int_{\frac{\beta}{\mu}>s_m} (S_0 - v)-\int_{\frac{\beta}{\mu}<s_m}  v \right)\\
&= s_m\left( \int_{\frac{\beta}{\mu}>s_m} S_0 - \int_\X v  \right)\\
&=0.
\end{array}
\end{equation*}
We used in the computations that, for any function $v\in \V^\star$, we have $v\leq S_0$ and that, by definition of $s_m$, we have $\int_{\frac{\beta}{\mu}>s_m} S_0 = m = \int_\X v$ here.

This concludes the proof: the function $v_m$ maximizes $\sigma_0$, hence it maximizes $\mathcal{N}^\star$, among all functions $v\in \V^\star$ such that $\int_\X v =m$.
\end{proof}


\begin{rmrk}
    We used in the proof the assumption that, for each $s\in \R$, the set $\{\frac{\beta}{\mu}=s\}$ is of measure zero. This assumption can be removed, up to defining the functions $v_m$ in the following way:
    $$
    v_m = S_0\left(\mathds{1}_{\frac{\beta}{\mu}> s_m}+c \mathds{1}_{\frac{\beta}{\mu}= s_m}\right),
    $$
    where $s_m = \inf\{t>0 \ : \ \int_{\frac{\beta}{\mu}>t}S_0>m\} $ and $c$ is such that $\int_\X v_m = m$.

    The rest of the proof would be similar, up to some additional computations in Step $2$ of Proposition \ref{prop m}.
\end{rmrk}

Proposition \ref{prop m} shows that, for a given $m\in[0,K]$ fixed, the vaccination function $v_m$ maximizes the number of surviving individuals in problem \eqref{ocp:equivalent}.

To fully prove Theorem \ref{th quali}, it remains to show that the vaccination strategy $v_K$ saves more individuals than the strategy $v_m$, for $m\leq K$.

From now on, we denote $S_\infty^m$ the limiting state for \eqref{syst 0} with $v=v_m$. It satisfies
$$
S_\infty^m(x) = (S_0(x) - v_m(x))\exp\left(\int_{\mathcal X}\frac{\beta(y)}{\mu(y)}\,S^m_\infty(y)\,\di y -\int_{\mathcal X}\frac{\beta(y)}{\mu(y)}\,(I_0(y) + S_0 (y)- v_m(y))\, \di y\right),
$$
hence
\begin{equation}\label{S infty m}
S_\infty^m(x) = \begin{cases}
    0 \quad \text{for}\ x \ \text{such that} \ \frac{\beta(x)}{\mu(x)}>s_m,\\
    S_0(x)e^{\sigma_\infty(v_m)-\sigma_0(v_m)}\eta
    ,\quad \text{elsewhere}.
\end{cases}
\end{equation}

\begin{proof}[Proof of Theorem \ref{th quali}]
    We have proven that
$$
\sup_{\substack{v\in \V^\star \\ \int_\X v =m }}\mathcal{N}^\star(v) =  \mathcal{N}^\star(v_m),
$$
where $v_m$ is given by \eqref{vI}.

Now, we want to prove that $v_K$ (that is, $v_m$ with $m=K$) reaches the maximum in \eqref{new opt I}. To do so, let us prove that $m\mapsto \mathcal{N}^\star(v_m)$ is nondecreasing. Let $0\leq m_1<m_2 \leq K$ be fixed. Let us prove that
$$
\mathcal{N}^\star(v_{m_1})\leq \mathcal{N}^\star(v_{m_2}),
$$
that is, we want to prove that
$$
m_1 + \int_X S_\infty^{m_1}\leq m_2 + \int_X S_\infty^{m_2}.
$$
We introduce the three following subsets of $\X$:
$$
A = \{x\in \X \ :\ v_{m_1}(x)  = S_0(x)\},\ B = \{x\in \X \ :\ v_{m_2}(x) = S_0(x) \ \text{and}\ v_{m_1}(x) = 0\} ,\ C=  \X \backslash(A\cup B).
$$
We could also write that $A=\{\frac{\beta}{\mu}>s_{m_1}\}$, $B=\{s_{m_2}<\frac{\beta}{\mu}\leq s_{m_1}\}$. 
Observe that, as $m_1<m_2$, we have $s_{m_1}>s_{m_2}$.

Owing to the shape of $v_{m_1},v_{m_2}$ and to \eqref{S infty m}, we have:
\begin{itemize}
    \item On the set $A$: $v_{m_1}=v_{m_2}=S_0>0$ and $S_\infty^{m_1}=S_\infty^{m_2}=0$.
    \item On the set $B$: $v_{m_1}=0$ and $v_{m_2}= S_0>0$ and $S_\infty^{m_2}=0$.
    \item On the set $C$:  $v_{m_1}=v_{m_2}=0$.
\end{itemize}
Let us add that we also have
$$
S_\infty^{m_1}\leq S_\infty^{m_2} \quad \text{on }\ C.
$$
Indeed, on the set $C$, for $k=1,2$, we have
$$
S_\infty^{m_k} = S_0 e^{\sigma_\infty(v_{m_k})-\sigma_0(v_{m_k})} \eta.
$$
This comes from the observation that, because $v_{m_2}\geq v_{m_1}$, we have $\sigma_0(v_{m_2})\leq\sigma_0(v_{m_1})$ and therefore, owing to Lemma \ref{lem sigma}, we have $\sigma_\infty(v_{m_2})-\sigma_0(v_{m_2})\geq \sigma_\infty(v_{m_1})-\sigma_0(v_{m_1})$.\\

Using all these remarks, we get
$$
\mathcal{N}^\star(v_{m_1}) = \int_A v_{m_1} + \int_B S_\infty^{m_1} + \int_C S_\infty^{m_1}
$$
and
$$
\mathcal{N}^\star(v_{m_2}) = \int_A v_{m_2} + \int_B v_{m_2} + \int_C S_\infty^{m_2}.
$$
On the set $A$, $v_{m_1} = v_{m_2}=S_0$. On the set $B$, we have $v_{m_2} = S_0 $ which is larger than $S_\infty^{m_1}$. Finally, we showed that on the set $C$, we have $S_\infty^{m_1}\leq S_\infty^{m_2}$.

Putting all that together, we find that $\mathcal{N}^\star(v_{m_1})\leq \mathcal{N}^\star(v_{m_2})$. This concludes the proof.
\end{proof}

\textbf{Acknowledgement. } E.P. acknowledges the support of Turin Polytechnic which partially funded a visiting period during the PhD that laid the foundation for the present work.
This study contributes to the IdEx Université de Paris ANR-18-IDEX-0001. The research leading to these results has received funding from the ANR project “ReaCh” (ANR-23-CE40-0023-01).
This work received support from the Maths-ArboV project which is funded by the PEPR Maths-Vives (managed by the National Research Agency (ANR), under the France 2030 program, reference ANR 23 EXMA 0001).

\bibliographystyle{siam}
\bibliography{Bibliografia_sirv}
\end{document}